\documentclass{article}

\usepackage{graphicx}

\usepackage[left=1in,right=1in,bottom=1in,top=1in]{geometry}
\usepackage{graphicx}
\usepackage{amsmath,amsthm,amssymb}
\usepackage[all,cmtip]{xy}
\usepackage{cite}
\usepackage{setspace}

\theoremstyle{plain}
\newtheorem{lem}{Lemma}
\newtheorem{lemma}[lem]{Lemma}

\newtheorem{theorem}[lem]{Theorem}
\newtheorem{prop}[lem]{Proposition}
\newtheorem{proposition}[lem]{Proposition}

\newtheorem{corollary}[lem]{Corollary}

\newtheorem{conjecture}[lem]{Conjecture}

\theoremstyle{definition}

\newtheorem{definition}[lem]{Definition}

\theoremstyle{remark}

\DeclareMathOperator{\oor}{or}
\DeclareMathOperator{\sg}{sg}

\DeclareMathOperator{\hht}{ht}
\DeclareMathOperator{\sep}{sep}
\DeclareMathOperator{\tr}{tr}

\title{A random walk on the symmetric group generated by random involutions}

\author{Megan Bernstein \\ Georgia Institute for Technology} 

\begin{document}

\maketitle


\begin{abstract}
 The involution walk is the random walk on $S_n$ generated by involutions with a binomially distributed with parameter $1-p$ number of $2$-cycles. This is a parallelization of the transposition walk. The involution walk is shown in this paper to mix for $\frac{1}{2} \leq p \leq 1$ fixed, $n$ sufficiently large in between $\log_{1/p}(n)$ steps and $\log_{2/(1+p)}(n)$ steps. The paper introduces a new technique for finding eigenvalues of random walks on the symmetric group generated by many conjugacy classes using the character polynomial for the characters of the representations of the symmetric group. Monotonicity relations used in the bound also give after sufficient time the likelihood order, the asymptotic order from most likely to least likely permutation. The walk was introduced to study a conjecture about a random walk on the unitary group from the information theory of black holes. 
\end{abstract}



\section{Introduction}

This paper examines how a natural notion of ``parallelization'' affects the rate of convergence to stationary for a random walk on the symmetric group. The base walk that this paper parallelizes is the $p$ lazy random transposition walk. It has as generators the identity with probability $p$ and a uniformly random transposition with probability $1-p$. This is equivalent to putting $n$ cards on the table and with probability $1-p$ swapping a random pair. The transposition walk for $p=\frac{1}{n}$ takes order $\frac{1}{2}n\log(n)+cn$ steps to converge to its uniform stationary distribution \cite{DS}. Suppose the walk is parallelized by simultaneously transposing $s$ disjoint pairs at the same time. This is like taking $s$ steps of the non-lazy transposition walk, except it guarantees $2s$ distinct cards are moved. This problem can be explored in several ways. For $n$ even, the maximum number of disjoint transpositions is $n/2$. If these are chosen via a random matching, a randomly chosen fixed point free involution results. The walk generated by all fixed point free involutions was analyzed by Lulov \cite{Lulov}, who showed that it mixes in $3$ steps. In this paper, each transposition in a fixed point free involution is discarded with some probability. This is a parallelized $p$-lazy transposition walk. When this probability is fixed and at least $\frac{1}{2}$ the results here show that this walk has mixing time $\Theta(\log(n))$.

More specifically, this paper studies the random walk on $S_n$, for $n$ even, generated by first choosing uniformly at random a fixed point free involution, also known as a perfect matching, then discarding or keeping each $2$-cycle it contains independently with probability $p$, $1-p$ respectively. This means the probability an involution with $s$ $2$-cycles is selected is ${{n/2 \choose s}} p^{n/2-s}(1-p)^{s}$. By considering general $p$, this gives a family of walks with $pn$ as the expected number of fixed points of a generator.  Taking $p=1-2/n$ gives an ``expected transposition walk'' where on average, a transposition will be selected, or as $p \rightarrow 0$, an expected fixed-point free involution walk. The author conjectures that mixing occurs with cutoff at $\log_{1/p}(n) + \log_{1/p}(c)$ steps for any $p$ bounded away from $0$. This mixing time would interpolate from an expected transposition walk to an expected $s$ $2$-cycle walk for any $s<n/2$ with comparable mixing times to their non-random cousins, in particular the transposition walk mixing with cutoff at $\frac{1}{2}n\log(n) + cn$ steps \cite{DS}. This paper, for $p\geq 1/2$ fixed, for $n$ sufficiently large, establishes for mixing a lower bound of $\log_{1/p}(n)$ in Theorem \ref{theorem:invlb} and an upper bound of $\log_{2/(1+p)}(n)$ in Theorem \ref{theorem:invup}. These are separated by just over a factor of $2$.

This upper bound is found through a combination of two methods. Both use the expression of the eigenvalues of the walk in terms of the characters of the symmetric group. The character polynomial gives the characters of $S_n$ as a polynomial in the cycle decomposition of a permutation. The eigenvalues of this walk, as seen in (\ref{eigf}), are a linear combination of characters evaluated at the $n/2+1$ conjugacy classes of involutions. Since all these involutions have only $1$- and $2$-cycles, understanding the character polynomial in these cycles will give a strong bound on the large eigenvalues of the walk.  A recursive formula for the eigenvalues given in Proposition \ref{eigformula} is used to control the small eigenvalues. This recursion constructed via the Murnaghan-Nakayama rule leads to a series of monotonicity conditions on the eigenvalues. These monotonicity conditions require that $p \geq \frac{1}{2}$.

A secondary result of these monotonicity conditions, and so also restricted to $p \geq \frac{1}{2}$, is a total order for the most likely to least likely element after sufficient time is identified in Corollary \ref{lhoiw}. At each step of a Markov chain, there is a partial ordering from the most likely to the least likely state. A linear extension to a total order is called a likelihood order. With mild conditions, after sufficient time, this converges to a fixed likelihood order. For the involution walk with $p \geq \frac{1}{2}$, the limiting likelihood order is the cycle lexicographic order as defined in Definition \ref{CLorder}. This means after sufficient time the identity will be the most likely element and an $n$-cycle the least likely element of the walk. This is the same likelihood order as $p$ lazy transposition walk for $p\geq \frac{1}{2}$ \cite{Mythesis}. Likelihood orders are motivated by the total variation distance and separation distance metrics of studying Markov chain convergence. 

The most common quantification of the convergence to uniform of a random walk on a group $G$ is total variation distance. Let $P^{*t}(g)$ denote the probability of being at $g$ at time $t$ for a random walk on $G$. Let $A \subset G$ denote a subset $A$ of $G$, and $P^{*t}(A)$ the total probability of elements of $A$. Then,
\[ \left|\left|P^{*t}(\cdot) - \frac{1}{|G|}\right|\right|_{TV} = \max_{A \subset G} \left|P^{*t}(A) - \frac{|A|}{|G|}\right| \]
One maximal set is $A = \{ g \in G| P^{*t}(g) > \frac{1}{|G|} \}$. Identifying where $\frac{1}{|G|}$ sits inside the likelihood order splits the group elements into this set $A$ and its compliment. The principal technique for finding likelihood orders will also give what this set is after sufficient time. Even a partial identification of this maximal set is of use in constructing lower bounds on mixing, in other words, for showing the total variation distance is not yet small.  Separation distance measures how much less likely than uniform the least likely element is. Identifying the least likely element means this can be computed directly. Separation distance can be measured indirectly through Strong Stopping Time arguments. 

The first work on likelihood orders was done by Diaconis and Graham, see Chapter 3C Exercise 10 of \cite{Diaconis}, motivated by a statistical problem posed by Tom Ferguson. Likelihood orders were also later studied by Diaconis and Isaacs \cite{DI}. They found that the random walk on a cycle or hypercube showed a strong monotonicity condition consistent with the total order given by distance as counted by the generators from the start. Since their likelihood orders hold at all times, simple induction suffices. For more examples of inductive proofs of likelihood orders see Chapter 1 of \cite{Mythesis}. 

The limiting likelihood order is a statement of eigenvalue monotonicity. One element is eventually more likely than another if, in the difference of their probabilities expressed in terms of the eigenvalues using the discrete Fourier inversion formula, the term with largest eigenvalue with non-zero coefficient is positive.  Diaconis and Shahshahani \cite{DS} in their seminal paper on mixing for the random transposition walk used a formula of Frobenius to establish a monotonicity property for the eigenvalues of the walk. The eigenvalues are labeled by partitions of $n$, and they observed that a classical partial order on partitions called majorization order is consistent with a monotonic decline in the eigenvalues. Diaconis and Graham showed that no total monotonic order can hold at all times for the transposition walk on the symmetric group. The likelihood order fluctuates within even a small number of steps (for $n \geq 6$, the first change occurs after four steps). Lulov, in his thesis \cite{Lulov}, connected the ordering on eigenvalues to what he termed an ``asymptotic monotonicity property'', here called a likelihood order, of the elements of the walk after sufficient time. He showed the transposition walk restricted to even steps after sufficient time followed a cycle lexicographic order. 

As $p \rightarrow 0$ and $n$ is held constant, due to a parity problem, the walk can no longer mix in $O(\log(n))$ steps let alone the smaller $O(\log_{1/p}(n))$ steps. The fixed point free involution walk at even steps is confined to $A_n$ inside of $S_n$. While for any $p >0$, the involution walk will mix to all of $S_n$, as $p \rightarrow 0$, the probability of selecting a fixed point free involution at each step of the walk will grow to $1$. Since it becomes more and more unlikely as $p \rightarrow 0$ anything other than a fixed point free involution is chosen, at even steps, the involution walk will be more and more prone to be stuck on even elements inside of $S_n$ and take longer and longer to approach uniformity over all elements. This is shown in Proposition \ref{tinyp}.

A random permutation can be made from at most $n$ transpositions chosen systematically. This systematic scan consists of transposing in order the number in each position with itself or a later position uniformly at random. If instead the transpositions are chosen uniformly at random, as in the transposition walk, a random permutation takes $\frac{1}{2}n\log(n) + cn$ transpositions to build ~\cite{DS}. The largest impediment is a coupon collector problem of never choosing a transposition containing a large fraction of numbers by $\frac{1}{2}n\log(n) -cn$ steps.  These never moved numbers are fixed points of the permutation, resulting in insufficiently random permutations. On the other hand, generating the random even permutation from a walk generated by fixed point free involutions, takes $3$ steps of the walk or $\frac{3}{2}n$ transpositions ~\cite{Lulov}. In the involution walk by letting $p$ vary, one can study this transition from the minimum of $O(n)$ transpositions to $O(n\log(n))$ to build a random permutation. The analysis here holds for fixed $p \geq \frac{1}{2}$ and $n$ sufficiently large. In all such cases it takes $O(n\log(n))$ transpositions to build a random permutation.


While studying the information theory of black holes, physicists became interested in a random walk on the unitary group. Information in the black hole is expressed in qubits, each an element of $\mathbb{C}^2$. Take as basis vectors $e_1 = (1,0)$, $e_0 = (0,1)$. The walk is on $n$ qubits, and so it on a $2^n$-dimensional space with a basis indexed by $n$-bit binary strings. At each step the walk takes a random $U(4)$ operator and applies it to two random qubits, acting on the binary strings indexing the basis. All $2^{n-2}$ basis vectors with the same $2$-bit combination for those two qubits are effected the same way under the walk. This means each $U(4)$ operator acts $2^{n-2}$ times, giving rapid mixing for such a high dimensional space. This walk is known to ``scramble'' in $n\log(n)$ steps \cite{SSblack}. Recent work of Hayden and Preskill \cite{HPblack} and Sekino and Susskind \cite{SSblack} has developed interest in a version of this walk where $n/2$ commuting steps of the random walk are taken at once. A different $U(4)$ operator is chosen to each act on the different $2$-cycles of a perfect matching of the qubits. This faster walk is conjectured to mix in $O(\log(n))$ steps. The involution walk was designed as a toy model to study the effects of independent random shuffling on the components of a perfect matching. 

Section \ref{background} describes the upper bound lemma from discrete Fourier analysis and the eigenvalues of the walk. Section \ref{mon} finds monotonic decay of the eigenvalues needed for the bounds and the likelihood order of the walk. Section \ref{biinv} finds an upper bound for the mixing time of the binomially distributed involution random walk. Section \ref{lowerbound} finds the the lower bound of $\log_{\frac{1}{p}}(n)$. Section \ref{conjsep}  calculates the separation distance for the involution walk assuming the conjecture that the likelihood order holds at all times.

\section{Background}\label{background}

The upper-bound lemma gives a bound on that the total-variation distance between the walk on a group and its uniform stationary distribution using its eigenvalues expressed in terms of the groups representations  ~\cite{Diaconis}. The version below is specialized to conjugacy class walks on the symmetric group, in which every element of a conjugacy class is equally likely. Partitions of $n$ index the non-trivial irreducible representations of the symmetric group as well as the eigenvalues of the walk; the representation for the partition $\lambda$ has dimension $d_{\lambda}$.

\begin{prop}[Diaconis-Shahshahani \cite{DS}]
When $K(t)$ is a class function of an aperiodic, irreducible walk on $S_n$,
\[ \left|\left|K^{*t}(\sigma) - \frac{1}{n!}\right|\right|_{TV} \leq \frac{1}{4} \sum_{\lambda \neq 1} d_{\lambda}^2 \psi_{\lambda}^{2t} \]
The sum below is over conjugacy classes $\kappa$ of size $|\kappa|$ with $K(\kappa)$ the probability of one element of the conjugacy class, \[\psi_\lambda = \sum_{\kappa } |\kappa|K(\kappa)\frac{\chi_{\lambda}(\kappa)}{d_\lambda}\] 
\end{prop}

For this walk, the formula for the eigenvalue $\psi_\lambda$ is the sum over conjugacy classes of the probability of it being a generator times its character ratio:

\begin{align}\label{eigf} \psi_\lambda = \sum_{s=0}^{n/2} p^{n/2-s}(1-p)^s {{n/2 \choose s}} \frac{ \chi_{\lambda}(1^{n-2s},2^{s})}{d_{\lambda}} \end{align}

Bounds for these eigenvalues will be established through a combination of monotonicity relations on the eigenvalues and upper bounds on a handful of the eigenvalues using the character polynomial.

One formula for the character for the representation indexed by $\lambda$ evaluated at a conjugacy class $\alpha$, $\chi_\lambda(\alpha)$ is given by the Murnaghan-Nakayama rule. This rule expresses the character in terms of all the ways to decompose the partition $\lambda$ into borderstrips (also known as rimhooks) of sizes $\alpha_1,...,\alpha_r$ in any fixed order. A borderstrip is a skew-partition, or the difference of two partitions, containing no two by two boxes. A decomposition can be written as a sequence of partitions $P= (\rho_0,...,\rho_r)$ where $\rho_0=\lambda$, $\rho_r = \emptyset$ and the difference between sequential partitions, $\rho_i/\rho_{i+1}$ each with $\alpha_{i+1}$ boxes, is a borderstrip. We say the height of a borderstrip is one less than the vertical height of that partition, and the height of $P$, $\hht(P)$, is the sum of the heights of the border strips in the decomposition. Since the walk is generated by only one and two cycles, only borderstrips of size $1$ and $2$ will be necessary. There is only one borderstrip of size one - corresponding to the partition $[1]$ or a single box in a Young diagram. The two borderstrips of size two are $[2]$, with Young diagram two horizonal boxes, and $[1,1]$, with Young diagram two verticle boxes. These three borderstrips have heights zero, zero, and one, respectively. 

\begin{prop}[Murnaghan-Nakayama Rule]
\[\chi_\lambda(\alpha) = \sum_{P=(\rho_0,...,\rho_r)} (-1)^{\hht{P}}\]
\end{prop}

\section{Monotonicity of Eigenvalues}\label{mon}

The eigenvalues of this walk show intriguing connections to the eigenvalues of the transposition walk. For example, in the transposition walk, the eigenvalues decrease according to majorization order, where $\lambda$ is smaller than $\rho$ if the blocks of $\lambda$ can be moved up and to the right to get $\rho$. Below, this pattern is shown for the eigenvalues of the involution walk when the eigenvalue pairs are restricted to any $\lambda$ and $\rho = [n-i,i]$. This monotonicity relation on eigenvalues is used in the upper bound section to get a bound for $\lambda$ with $\lambda_1 < \frac{n}{2}$. The result also gives that the likelihood order after sufficient time is the cycle lexicographical order, the same order as for the transposition walk.

This section is based upon the following recursive construction of the eigenvalues. This is derived from a probabilistic Murnaghan-Nakayama rule. In the Murnaghan-Nakayama rule, the sizes of the borderstips are the sizes of the cycles in the conjugacy class. The decomposition can be done with any ordering of the these sizes. Since the cycles decomposition of the generators of this walk are probabilistic, the ordering of the sizes can as well. The following formula comes from decompositions in which the first borderstrips in the decomposition are two $1$-cycles with probability $p$ or a two-cycle with probability $1-p$. Examining all the configurations of borderstrips that can be removed and their heights amounts to:
\begin{prop}\label{eigformula} For $\psi_\lambda$ defined in (\ref{eigf}),
\[ \psi_\lambda = \sum_{\rho: \lambda/ \rho = [2]} \psi_\rho \frac{d_\rho}{d_\lambda} + (2p-1)\sum_{\rho: \lambda/ \rho = [1,1]} \psi_\rho \frac{d_\rho}{d_\lambda} + 2p\sum_{\rho: \lambda/ \rho = [1]\cup[1]} \psi_\rho \frac{d_\rho}{d_\lambda}\]
\end{prop}
\begin{proof}
Recall, the involutions are generated by starting with a perfect matching and removing transpositions with probability $p$. In a generator the first transposition from the starting perfect matching remains with probability $1-p$ or becomes two fixed points with probability $p$. The single border strip of size one is $[1]$ while the borderstrips of size two are $[2]$ and $[1,1]$, with height $0$ and $1$ respectively. By the Murnaghan-Nakayama rule, 

\begin{align*} \psi_\lambda =& \sum_{s=0}^{n/2} \frac{\chi_\lambda(1^{n-2s},2^s)}{d_\lambda}p^{n/2-s}(1-p)^{s}{n/2 \choose s} \\
=&\sum_{s=0}^{n/2}(1-p)\sum_{\rho: \lambda/\rho = [2] \oor [1,1] } (-1)^{\sg (\lambda/\rho)}\frac{\chi_{\rho}(1^{n-2s},2^{s-1})}{d_{\rho}}\frac{d_{\rho}}{d_\lambda}p^{\frac{n-2}{2}-(s-1)}(1-p)^{s-1}{\frac{n-2}{2} \choose s-1} \\
&+ p\sum_{\lambda \subset \gamma \subset \rho, |\lambda/\gamma| =1, |\gamma/\rho|=1 }\frac{\chi_{\rho}(1^{n-2s-2},2^{s})}{d_{\rho}}\frac{d_{\rho}}{d_\lambda}p^{\frac{n-2}{2}-s}(1-p)^{s}{\frac{n-2}{2} \choose s}\\
=& \sum_{\rho: \lambda/ \rho = [2]} \psi_\rho \frac{d_\rho}{d_\lambda} + (2p-1)\sum_{\rho: \lambda/ \rho = [1,1]} \psi_\rho \frac{d_\rho}{d_\lambda} + 2p\sum_{\rho: \lambda/ \rho = [1]\cup[1]} \psi_\rho \frac{d_\rho}{d_\lambda}\end{align*}
\end{proof}

To show monotonicity conditions on the eigenvalues through the recursive definition, it will be shown that the sum of these three terms is larger for one partition than another. The monotonicity does not follow term by term, only collectively. The following observation will be useful in the arguments that follow.

\begin{prop}\label{eigmaj}
The sum of the coefficients in the expansion of $\psi$ in Proposition \ref{eigformula} decreases according to majorization order.
\end{prop}
\begin{proof}

\begin{align*} \psi_\lambda =& \sum_{\rho: \lambda/ \rho = [2]} \psi_\rho \frac{d_\rho}{d_\lambda} + (2p-1)\sum_{\rho: \lambda/ \rho = [1,1]} \psi_\rho \frac{d_\rho}{d_\lambda} + 2p\sum_{\rho: \lambda/ \rho = [1]\cup[1]} \psi_\rho \frac{d_\rho}{d_\lambda}\end{align*}

Examining this without the $\psi_\rho$ terms, 
\begin{align*} &\sum_{\rho: \lambda/ \rho = [2]} \frac{d_\rho}{d_\lambda} + (2p-1)\sum_{\rho: \lambda/ \rho = [1,1]} \frac{d_\rho}{d_\lambda} + 2p\sum_{\rho: \lambda/ \rho = [1]\cup[1]} \frac{d_\rho}{d_\lambda} \\
= &p \sum_{\gamma: |\lambda / \gamma| = 1, |\gamma / \rho|=1} \frac{d_\rho}{d_\lambda} + (1-p)\sum_{\rho: \lambda/\rho = [2] \oor [1,1]} (-1)^{\sg(\lambda/\rho)} \frac{d_\rho}{d_\lambda}\\
=&p + (1-p)\frac{\chi_\lambda(\tau)}{d_\lambda} \end{align*}

Where the equality is by the Murnaghan-Nakayama rule. Moreover, a classical result of Frobenius \cite{DS}, shows $\frac{\chi_\lambda(\tau)}{d_\lambda}$ decreases along majorization order. For a formula for $\frac{\chi_{\lambda}(\tau)}{d_\lambda}$ see \cite{DS}. 
\end{proof}

\begin{lemma}\label{deci}
For $p \geq \frac{1}{2}$, $\psi_{[n-i,i]}$ for $i \leq n/2$ decreases as $i$ increases.
\end{lemma}
\begin{proof}
Using the formula \[ \psi_\lambda = \sum_{\rho: \lambda/ \rho = [2]} \psi_\rho \frac{d_\rho}{d_\lambda} + (2p-1)\sum_{\rho: \lambda/ \rho = [1,1]} \psi_\rho \frac{d_\rho}{d_\lambda} + 2p\sum_{\rho: \lambda/ \rho = [1]\cup[1]} \psi_\rho \frac{d_\rho}{d_\lambda} \]
by induction on $n$, it will follow that $\psi_{[n-i,i]} \geq \psi_{[n-i-1,i+1]}$. The base case of $n=2$ has $\psi_{[2]}=1 \geq \psi_{[1,1]} = 2p-1$, since $\psi_{[1,1]} = p\frac{\chi_{[1,1]}(1^2)}{d_{[1,1]}} + (1-p)\frac{\chi_{[1,1]}(2)}{d_{[1,1]}} = p - (1-p) = 2p-1$. 

Fix $n$ and assume the lemma holds for partitions of $n-2$. For $i \leq \frac{n}{2} -1$ the eigenvalue decomposes using the hook length formula to:

\begin{align*} \psi_{[n-i,i]} =& \frac{(n-i+1)_2}{(n)_2}\left(1 - \frac{2}{n-2i+1}\right)\psi_{[n-i-2,i]} + \frac{(i)_2}{(n)_2}\left(1 + \frac{2}{n-2i+1}\right)\psi_{[n-i,i-2]} \\
&+ p \frac{2i(n-i+1)}{(n)_2}\psi_{[n-i-1,i-1]} \end{align*}

The sizes of these three eigenvalues for the involution walk on $S_{n-2}$ and their coefficients will be compared to those appearing for the partition $[n-i-1,i+1]$.

Case 1: For $i \leq n/2-2$, for the eigenvalue for the partition $[n-i-1,i+1]$ this becomes,
\begin{align*} \psi_{[n-i-1,i+1]} =&\frac{(n-i+2)_2}{(n)_2}\left(1 - \frac{2}{n-2i-1}\right)\psi_{[n-i-3,i+1]} + \frac{(i+1)_2}{(n)_2}\left(1 + \frac{2}{n-2i-2}\right)\psi_{[n-i-1,i-1]} \\ &+ p \frac{2(i+1)(n-i)}{(n)_2}\psi_{[n-i-2,i]} \end{align*}

Two of the eigenvalues of $n-2$ that appear in $[n-i,i]$ are larger by induction than those in $[n-i-1,i+1]$. The exception being $[n-i-2,i]$ in the decomposition of $[n-i,i]$ has a longer second row than $[n-i-1,i-1]$ in the decomposition of $[n-i-1,i+1]$, which means by induction that its a larger eigenvalue. This means its enough to check the coefficients of $[n-i,i-2]$ and $[n-i-1,i-1]$ from $[n-i,i]$ are larger than the coefficient of $[n-i-1,i-1]$ from $[n-i-1,i+1]$, and that the sum of all three coefficients in the $[n-i,i]$ expression are larger than those in the $[n-i-1,i-1]$. The latter holds by Proposition \ref{eigmaj}. For the former this amounts to:

\[ i(i+1) + \frac{2i(i+1)}{n-2i-1} \leq i(i-1) + \frac{2i(i-1)}{n-2i+1} + 2p(n-i+1)i \]

Which simplifies to $p(n-i+1) \geq \frac{i}{(n-2i+1)(n-2i-1)} + 1 + \frac{1}{n-2i-1} + \frac{1}{n-2i+1}$

The left side is decreasing with $i$, while all terms on the right increase with $i$, so its enough to consider $i = n/2-2$ and $p=\frac{1}{2}$, in which case it is true that,

\[ \frac{1}{2}\left(\frac{n}{2} + 3\right) \geq \frac{\frac{n}{2} -2}{15} + 1 + \frac{1}{3} + \frac{1}{5}\]

Case 2: This leaves $i = n/2-1$ and $i+1 = n/2$, where,
\[\psi_{[n/2+1,n/2-1]} = \frac{5}{12}\frac{(n-2)(n-4)}{(n)_2}\psi_{[n/2+1,n/2-3]} + \frac{1}{2}p \frac{(n-2)(n+4)}{(n)_2}\psi_{[n/2,n/2-2]} + \frac{1}{12}\psi_{[n/2-1,n/2-1]}\]
\[ \psi_{[n/2,n/2]} = \psi_{[n/2,n/2-2]}\frac{3}{4}\left(1 - \frac{1}{n-1}\right) + \psi_{[n/2-1,n/2-1]}(2p-1)\frac{1}{4}\left(1 + \frac{3}{n-1}\right)\]
In the terms above, the $[n/2+1,n/2-3]$ and $[n/2,n/2-2]$ in $\psi_{[n/2+1,n/2-1]}$ overcomes the analogous term $[n/2,n/2-2]$ in $\psi_{[n/2,n/2]}$ for $p \geq \frac{1}{2}$ since,  \[\frac{5}{12}\frac{(n-2)(n-4)}{(n)_2} + \frac{1}{4} \frac{(n-2)(n+4)}{(n)_2} \geq \frac{3}{4}(1 - \frac{1}{n-1})\] 

This shows the the sum of the coefficients of the first two terms in the decomposition of $[n/2+1,n/2-1]$ is larger than the sum of the first coefficient in the decomposition of $[n/2,n/2]$. This was done with the two terms with longer first rows than the one term, $[n/2,n/2-2]$. It remains to show that the sum of all three coefficients in $[n/2+1,n/2-1]$ is greater than the sum of both coefficients in $[n/2,n/2]$. This once again follows by Proposition \ref{eigmaj}.

\end{proof}

\begin{lemma}\label{twopartbiggest}
For all $\lambda$ with $\lambda_1 = n-i$, $p \geq \frac{1}{2}$, \[ \psi_\lambda \leq \psi_{[n-i,i]}\]
\end{lemma}
\begin{proof}
The proof will follow by induction on $n$. The base case for $n=2$ is trivial as $[2]$ and $[1,1]$ are the only partitions with their first rows. Suppose it holds for all $i$ for $n-2$. If $i < \frac{n}{2}$, no vertical removals per Murnaghan-Nakayama are possible from the first row and, 

\[ \psi_{[n-i,i]} = \frac{d_{[n-i-2,i]}}{d_{[n-i,i]}}\psi_{[n-i-2,i]} + \frac{d_{[n-i,i-2]}}{d_{[n-i,i]}}\psi_{[n-i,i-2]} + 2p \frac{d_{[n-i-1,i-1]}}{d_{[n-i,i]}}\psi_{[n-i-1,i-1]} \]

There are three cases of lengths of first row this can generate: $n-i$,$n-i-1$ and $n-i-2$. Call the coefficients in the decomposition of $[n-i,i]$ of these terms $a_{n-i}$,$a_{n-i-1}$ and $a_{n-i-2}$. Let $\lambda = [n-i,...]$ be another partition. Let $b_{n-i}$ be the sum of cofficients of any $[n-i,...]$ in its decomposition, similarly define $b_{n-i-1}$ and $b_{n-i-2}$. By Lemma \ref{deci}, the corresponding eigenvalues increase as the first row increases. It needs to be shown that $a_{n-i} \geq b_{n-i}$, $a_{n_i} + a_{n-i-1} \geq b_{n-i} + b_{n-i-1}$ and finally $a_{n-i} + a_{n-i-1} + a_{n-i-2} \geq b_{n-i} + b_{n-i-1} + b_{n-i-2}$. This last equation holds by Proposition \ref{eigmaj}. The strategy to show the other is to first show that $a_{n-i} \geq b_{n-i}$; in other words that a removal all from below the first row is more likely $[n-i,i]$ than $\lambda$. The last inequality is shown in indirectly, by finding that $a_{n-i-2} \leq b_{n-i-2}$; the probability of a $2$ cycle removal from the first row, which gives the shortest first row, is more likely than the same removal in $[n-i,i]$. Then $a_{n-i} + a_{n-i-1} \geq b_{n-i} + b_{n-i-1} + b_{n-i-2} - a_{n-i-2} \geq b_{n-i} + b_{n-i-1}$.

Consider removing two blocks from below the first row. This effects at most two hook lengths from the first row. The smallest such hook lengths its possible to effect occur in $[n-i,i]$, causing the largest increase to the ratio of old versus new contributions of the first row. Let $h_{1,j}'$ denote the new hook lengths of the first row. Let $\rho$ be a partition of $i-2$ obtained by removing the first row of $\lambda$ and two additional squares. Then,
\[ \frac{d_{n-i,\rho}}{d_{n-i,\lambda/\lambda_1}} = \frac{\frac{(n-2)_{\lambda_1}}{\prod h_{1,j}'} d_\rho}{ \frac{(n)_{\lambda_1}}{\prod h_{1,j}} d_{\lambda/\lambda_1}} = \frac{(i)_2}{(n)_2} \frac{\prod h_{1,j}}{\prod h_{1,j}'}\frac{d_\rho}{d_{\lambda/\lambda_1}} \]
When $p=1$, the sum over all such $\rho$ of $\frac{d_\rho}{d_{\lambda/\lambda_1}}$ is $1$. For $[n-i,i]$ this is always $1$. As observed, $[n-i,i]$ also maximizes $\frac{\prod h_{1,j}}{\prod h_{1,j}'}$. This gives that the coefficient of $[n-i,i-2]$ is larger than the sum of coefficients for all two block removals below the first row of $\lambda$.

For both $\lambda$ and $[n-i,i]$ there is exactly one way to remove $2$ blocks from the first row. It must be shown that $\frac{d_{[n-i-2,\lambda/\lambda_1]}}{d_\lambda} \geq \frac{d_{[n-i-2,i]}}{d_{[n-i,i]}}$. 
\begin{align} \label{casen-i-2}\frac{d_{[n-i-2,\lambda/\lambda_1]}}{d_\lambda} = \frac{2\prod_{j=1}^{n-i-2} \frac{h_{1,j}}{h_{1,j}-2}}{(n)_2} = \frac{\prod_{j=1}^{n-i-2} \left(1- \frac{2}{h_{1,j}}\right)^{-1}}{(n)_2}\end{align}
Consider  $\prod_{j=1}^{n-i-2} \left(1- \frac{2}{h_{1,j}}\right)$. Where $h_{1,j} = n-i-2 -j + \lambda_j'$, with $\sum \lambda_j' = n-2$ and $\lambda_j$ decreasing. This gives an optimization problem with bounded region and linear constraints. The maximal solution without the bounded region is to have all $h_{1,j}$ be equal. Given the constraints, the optimal solution is to take the $h_{1,j}$ as close to equal as possible giving $[n-i,i]$. This product is maximized at $[n-i,i]$, which in turn minimizes the ratio in \ref{casen-i-2}. Therefore, the coefficient of $[n-i-2,i]$ is smaller than that of a two block removal from the first row of any other $\lambda$ with $\lambda_1 = n-i$.

When $i= \frac{n}{2}$,
\[ \psi_{[n/2,n/2]} = \psi_{[n/2,n/2-2]}\frac{d_{[n/2,n/2-2]}}{d_{[n/2,n/2]}} + (2p-1)\psi_{[n/2-1,n/2-1]}\frac{d_{[n/2-1,n/2-1]}}{d_{[n/2,n/2]}}\]
The argument above for removing two blocks below the first row holds to show $\frac{d_{[n/2,n/2-2]}}{d_{[n/2,n/2]}}$ is larger than the probability $\lambda$ does the same. The probability $[n/2,n/2]$ decomposes to at least $[n/2-1,n/2-1]$ is larger than for any other $\lambda = [n/2,...]$ by Proposition \ref{eigmaj}.
\end{proof}

\begin{theorem}\label{n2bound}
For $p \geq \frac{1}{2}$ and $\lambda$ such that $\lambda_1' \leq \lambda_1 < \frac{n}{2}$, $\psi_{\lambda} \leq \psi_{[n/2,n/2]}$
\end{theorem}
\begin{proof}
By Proposition \ref{eigmaj} $[n/2,n/2]$ has the maximal value of the sum of coefficients in the formula from Proposition \ref{eigformula} over all $\lambda$ with $\lambda_1,\lambda_1'\leq n/2$ with the value $\frac{3}{4}(1 - \frac{1}{n}) + \frac{1}{4}(2p-1)(1 + \frac{3}{n-1})$. In the case that $p=\frac{1}{2}$ this value is $\frac{3}{4}(1 - \frac{1}{n-1})$. This gives a heuristic for why the eigenvalue of $[n/2,n/2]$ will be computed when $p=\frac{1}{2}$ as roughly $(3/4)^{n/2}$.

The bound on $\psi_\lambda$ is by induction. Since $p \geq \frac{1}{2}$ the only possibly negative terms in the expansion \[\psi_\lambda = \sum_{\rho: \lambda/ \rho = [2]} \psi_\rho \frac{d_\rho}{d_\lambda} + (2p-1)\sum_{\rho: \lambda/ \rho = [1,1]} \psi_\rho \frac{d_\rho}{d_\lambda} + p\sum_{\rho: \lambda/ \rho = [1]\cup[1]} \psi_\rho \frac{d_\rho}{d_\lambda}\] are the $\psi_\rho$. An upper bound on $|\psi_\rho|$ suffices to pull it in front of the expression.

The first $n$ with partition $\lambda$ with $\lambda_1' \leq \lambda_1 < \frac{n}{2}$ existing is $n=8$, with a single partition $\lambda = [3,3,2]$. $\psi_{[3,3,2]}\leq \psi_{[4,4]}$ as for each $s$, $0 \leq \frac{\chi_{[3,3,2]}(2^s)}{d_{[3,3,2]}} \leq \frac{\chi_{[4,4]}(2^s)}{d_{[4,4]}}$. This can be seen in the character table of $S_8$ \cite{Littlewood}. Then assume by induction the bound holds for $\rho$ a partition of $n-2$ with $\rho_1' \leq \rho_1 < \frac{n-4}{2}$, and it was proven before for $\rho_1 \leq \frac{n-2}{2}$ as well. The $\rho$'s that appear in the expression have $\rho_1 \leq \lambda_1 \leq \frac{n-2}{2}$, so $\psi_\rho \leq \psi_{[(n-1)/2,(n-1)/2]}$. Finally, $\left(\frac{3}{4}(1 - \frac{1}{n-1}) + (2p-1)\frac{1}{4}(1 + \frac{3}{n-1})\right)\psi_{[(n-1)/2,(n-1)/2} \leq \psi_{[n/2,n/2]}$ since \[ \psi_{[n/2,n/2]} = \psi_{[n/2,n/2-2]}\frac{d_{[n/2,n/2-2]}}{d_{[n/2,n/2]}} + \psi_{[n/2-1,n/2-1]}(2p-1)\frac{d_{[n/2-1,n/2-1]}}{d_{[n/2,n/2]}}\] And, since $\frac{d_{[n/2,n/2-2]}}{d_{[n/2,n/2]}} + (2p-1)\frac{d_{[n/2-1,n/2-1]}}{d_{[n/2,n/2]}} = \frac{3}{4}(1 - \frac{1}{n-1}) + (2p-1)\frac{1}{4}(1 + \frac{3}{n-1})$ and $\psi_{[n/2,n/2-2]} \geq \psi_{[n/2-1,n/2-1]}$ by Lemma \ref{deci}. Therefore, $\psi_\lambda \leq \psi_{[n/2,n/2]}$.

\end{proof}

These monotonicity relations on the eigenvalues, which will be used to form a upper bound on the mixing time of the walk, also show a quite granular structure inside of the mixing. After many steps, the walk will approach being uniform on all permutations, but some permutations will always be more likely than others. The likelihood order for the walk is a total order that describes the relative likelihood of the permutations. For the $p \geq \frac{1}{2}$ the transposition walk and involution walk, the likelihood order after sufficient time is the cycle lexicographic order on permutations defined as:

\begin{definition}\label{CLorder} Let $\alpha=(1^{a_1},2^{a_2},...,n^{a_n})$ where $\alpha$ as $a_i$ $i$-cycles, $\beta = (1^{b_1},2^{b_2},...,n^{b_n})$ be two conjugacy classes of $S_n$. Define the cycle lexicographic order to be $\alpha >_{CL} \beta$ when for $\min_k(a_k \neq b_k) = i$, $a_i > b_i$. 
\end{definition}

\begin{corollary}\label{lhoiw}
The likelihood order for this walk for $p \geq \frac{1}{2}$ and for $t$ sufficiently large is the cycle lexicographical order.
\end{corollary}
\begin{proof}
The difference in likelihood of two permutations $\alpha$ and $\beta$ can be studied through the discrete Fourier transform. For the involution walk at two permutations $\alpha$ and $\beta$,
\[ P^{*t}(\alpha) - P^{*t}(\beta) = \frac{1}{n!}\sum_{\lambda} \left( \chi_{\lambda}(\alpha) - \chi_{\lambda}(\beta) \right) d_{\lambda} (\psi_\lambda)^t \]

The trivial representation has eigenvalue and coefficient one in the discrete Fourier decomposition for both $\alpha$ and $\beta$ and so vanishes. Other partitions for which $\chi_{\lambda}(\alpha) = \chi_{\lambda}(\beta)$ will also not contribute to this quantity. After sufficient time, the terms for the partitions with largest eigenvalue in magnitude with $\chi_{\lambda}(\alpha) \neq \chi_{\lambda}(\beta)$ will be exponentially larger than any other terms and hense will will determine the sign of $P^{*t}(\alpha) - P^{*t}(\beta)$. In lazy walks the largest eigenvalue in magnitude almost always occurs for a single partition.

From \cite{Mythesis}, a partition is called an $i$-cycle detector if $\lambda_2 + \lambda_1' -2 \geq i$ and $\lambda_1 + \lambda_2' -2 \geq i$. If $\lambda$ is not an $i$-cycle detector and the smallest cycle differing in the cycle decomposition of $\alpha$ and $\beta$ is an $i$-cycle, then $\chi_{\lambda}(\alpha) - \chi_{\lambda}(\beta) =0$ \cite{Mythesis}. Therefore, one must only examine the $i$-cycle detecting partitions for each value of $i$ from $1$ to $n/2$ in order to find the eventual likelihood order. By Lemma \ref{deci}, Lemma \ref{twopartbiggest}, and Theorem \ref{n2bound}, the partition $[n-i,i]$ has the largest magnitude of eigenvalue of all $i$-cycle detecting partitions. More over, when $\alpha$ and $\beta$ first differ at an $i$-cycle,  $\chi_{[n-i,i]}(\alpha) - \chi_{[n-i,i]}(\beta) =a_i -b_i\neq 0$ \cite{Mythesis}. In this case, the term for $[n-i,i]$ in the discrete Fourier transform, $(\chi_{[n-i,i]}(\alpha) - \chi_{[n-i,i]}(\beta))d_\lambda \psi_{[n-i,i]}^t$, determines the sign of $P^{*t}(\alpha) - P^{*t}(\beta)$ for sufficiently large $t$. Since $\chi_{[n-i,i]}(\alpha) - \chi_{[n-i,i]}(\beta) =a_i -b_i$ in this case and $\psi_{[n-i,i]}>0$, the permutation with more $i$-cycles is more likely after sufficient time. This is the cycle lexicographic order from Definition \ref{CLorder}.  

\end{proof}

\section{Upper Bound on Mixing}\label{biinv}

This section will be working towards bounds on $\psi_\lambda$ to use in the upper bound formula,

\[ \left|\left|K^{*t}(\sigma) - \frac{1}{n!}\right|\right|_{TV}^2 \leq \frac{1}{4} \sum_{\lambda \neq 1} d_{\lambda}^2 \psi_{\lambda}^{2t} \]

Recall that,

\[ \psi_\lambda = \sum_{s=0}^{n/2} p^{n/2-s}(1-p)^s {{n/2 \choose s}} \frac{ \chi_{\lambda}(1^{n-2s},2^{s})}{d_{\lambda}} \]

Instead of bounding $\chi_{\lambda}(1^{n-2s},2^s)$ for each $s$ individually, the character polynomial will give an expression for the character as a polynomial in $n-2s$ and $s$. The character polynomial, $q_{\rho}(x_1,...,x_k)$ for the partition $\rho$ of $k$ is a polynomial in variables $x_1,...,x_k$ so that 
\[\chi_{[n-k, \rho_1,...,\rho_r]}(1^{x_1},...,k^{x_k},...,n^{x_n}) = q_{\rho}(x_1,...,x_k) \]
 for any conjugacy class $(1^{x_1},...,n^{x_n})$ of $S_n$. Garsia and Goupil ~\cite{CP} give a formula for the character polynomial akin to the Murnaghan-Nakayama rule run backwards from its traditional order, peeling off border strips of the largest cycles first.

\[ q_{\rho}(x_1,...,x_i,0,...,0) = \sum_{j} {{x_i \choose j}} \sum_{P = (\rho^0,...,\rho^{j})} (-1)^{ht(P)}q_{\rho^j}(x_1,...,x_{i-1},0,...,0) \]

Where $P$ ranges over all possible ways of removing border strips of size $i$ from $\rho$ so that a Young diagram remains at each step, as in Murnaghan-Nakayama. The formula says, choose $j$ $i$-cycles of the $x_i$ $i$-cycles and attempt to peel them off from below the first row of $\lambda$, and take the remaining $x_i-j$ $i$-cycles from the first row of $\lambda$. Recurse on the remaining shape with the next largest cycle size. In Murnaghan-Nakayama, the first row does not receive this special treatment. Letting $i=2$ gives the character polynomial for an involution as:

\[q_{\rho}(n-2s,s,0,...0) = \sum_{j} {{s \choose j}}  \sum_{P=(\rho^0,...,\rho^{j})} (-1)^{ht(P)}q_{\rho^j}(n-2s,0,...,0) \]\label{chp}

Where the last term can be expanded as \[q_{\rho^j}(n-2s) = d_{n-2s-|\rho^j|,\rho^j} = {{n-2s \choose |\rho^j|}}d_{\rho^j} \prod_{k=1}^{\rho^j_1}\frac{n-2s-|\rho^j|-k+1}{n-2s-|\rho^j|-k+(\rho^j)_k'+1}\]

Then an upper bound on $q_{\rho^j}$ that is more computationally tractable comes from ignoring the sign associated with the insertions, rounding $\frac{n-2s - |\rho^j|-k-i+1}{n-2s - |\rho^j|-k-i+(\rho^j)_i'+1}$ to $1$, and upper bounding the ways of inserting one and two cycles by the dimension of $\rho$ giving:

\[q_{\rho}(n-2s,s) \leq \sum_{j} {{s \choose j}}\sum_{\rho^0,...,\rho^{j}}{{n-2s \choose |\rho| -2j}} d_{\rho^j} \leq \sum_j {{s \choose j}}{{n-2s \choose |\rho|-2j}} d_{\rho} \]

Then using this in $\psi_{\lambda}$ and splitting $s$ into $j_1$ and $j_2$ gives:

\begin{align} \psi_\lambda \leq \sum_{s=0}^{n/2} p^{n/2-s}(1-p)^s {{n/2 \choose s}}\sum_{j}{{s \choose j}}{{n-2s \choose n- \lambda_1 - 2j}} \frac{ d_{\lambda / \lambda_1}}{d_{\lambda}} \\ =\frac{d_{\lambda / \lambda_1}}{d_{\lambda} }p^{n/2} \sum_{j_1, j_2} \left(\frac{1-p}{p}\right)^{j_1 + j_2} {{n/2 \choose j_1,j_2}}{{n-2j_1-2j_2 \choose n- \lambda_1 -2j_2}} \end{align}\label{js}

This says to approximate $\psi$ take the expectation over the binomial distribution over all ways to choose $j_1$ and $ j_2$ of the $n/2$ $2$-cycles to insert into the first row and the remaining partition and to split the remaining unused numbers into either the first row or the remaining partition. The $d_{\lambda / \lambda_1}$ factor takes into account that there may be may ways to arrange things in the lower part of the partition. When $\lambda_1 \geq n/2$, the maximum value of $\frac{ d_{\lambda / \lambda_1}}{d_{\lambda}}$ occurs at the partition $[n-i,i]$ where this a very good approximation. 

\begin{proposition}
For $1\leq i \leq \frac{n}{2}$, \[\max_{\lambda: \lambda_1=n-i}  \frac{ d_{\lambda / \lambda_1}}{d_{\lambda}} = \frac{d_{[i]}}{d_{[n-i,i]}} = {n \choose i}^{-1}\frac{n-i+1}{n-2i+1} \]
\end{proposition}
\begin{proof}
Using the hook length formula \cite{Stanley2},
\[\frac{d_{\lambda/\lambda_1}}{d_{\lambda}} = \frac{\prod_{k=1}^{\lambda_1} \lambda_1 - k + \lambda_k'}{(n)_{\lambda_1}} = {{n \choose \lambda_1}} \prod_{k=1}^{\lambda_1}\frac{\lambda_1 - k + \lambda_k'}{\lambda_1 -k +1} \]
The $\lambda_k'$ are decreasing, and the product is maximized if these are taken to be as even as possible. So for $1 \leq k \leq i$, $\lambda_k' =1$, for $k>i$, $\lambda_k'=0$. This is the partition $[n-i,i]$. 
\end{proof}

The bound used above on the character polynomial, principly that $\sum_{P = \{\rho_0,...,\rho_j\}} (-1)^{P} \leq d_{\rho}$, was sufficiently strong for the partitions with first row at least $n/2$, but not for those with smaller first row. However, by Proposition \ref{n2bound}, the eigenvalues for $\lambda$ with $\lambda_1 < n/2$ are bounded by the eigenvalue for $[n/2,n/2]$.  

The next step is to handle the sum (\ref{js}). Instead of counting how the the two cycles $(1,2),...,(n-1,n)$ and unchosen cycles used as fixed points are arranged separately, an easier approach exists. Consider instead, splitting the numbers $1,2,...,n$ into two parts. When $2i-1$ and $2i$ are in the same part, this could have happened using them as a single two cycle, or separately as fixed points, for a total weight under the binomial distribution of $\frac{1-p}{p} +1 = \frac{1}{p}$. And when $2i-1$ and $2i$ are not in the same part, this could only have happened from $1$-cycle insertion but two different ways, for a weight of $2$. 

\begin{proposition}\label{SEAWORLD}
\[ \sum_{j_1,j_2} \left(\frac{1-p}{p}\right)^{j_1 + j_2} {{n/2 \choose j_1,j_2}}{{n-2j_1-2j_2 \choose n- \lambda_1 -2j_1 -2j_2}} = \sum_{j} 2^j \frac{1}{p^{n/2-j}}{{n/2 \choose j, \frac{n-i-j}{2},\frac{i-j}{2}}} \]
\end{proposition}

Note that $j$ must be such that $i-j$,$n-i-j$ are both even. So, 

\begin{proposition}
For $\lambda_1 \geq \frac{n}{2}$ let $i = n - \lambda_1$, for $\lambda_1 \leq \frac{n}{2}$ let $i = \frac{n}{2}$ 
\[ \psi_{\lambda} \leq {{n \choose i}}^{-1} \frac{n-i+1}{n-2i+1}  \sum_{j \leq i, i-j \text{even}} (2p)^j {{n/2 \choose j, \frac{n-i-j}{2},\frac{i-j}{2}}} \]
\end{proposition}

Now to approximate the sum, one can use that it is less than $i/2$ times its largest term, except for $i$ small where the largest term is the last and the other terms will be exponentially smaller. Note that nothing is assumed about $p$ in this bound.

\begin{proposition}
When \[ \alpha= \sqrt{1 + \frac{1-p^2}{p^2}\frac{4(n-i)(i)}{n^2}}\]
\begin{align*} \sum_{j} (2p)^j {{n/2 \choose j, \frac{n-i-j}{2},\frac{i-j}{2}}} \leq& {n/2 \choose i/2}\left(1 - \frac{n-i}{n}\frac{2}{1+ \alpha}\right)^{-\frac{i}{2}}\left(1 - \frac{i}{n}\frac{2}{1 + \alpha}\right)^{-\frac{n-i}{2}}\\
&\frac{i}{2}\left(1_{i = n/2}\frac{1}{\sqrt{2\pi i}} + 1_{i\neq n/2}\sqrt{\frac{n-2i+1}{n-i+1}}  \right)\end{align*}
\end{proposition}
\begin{proof}
One version of Stirling's formula is that \begin{align} \label{Stirling} \sqrt{2n}\left(\frac{n}{e}\right)^{n/e} \leq n! \leq \sqrt{n}e\left(\frac{n}{e}\right)^{n/e} \end{align} Applying the lower bound to $j! \frac{n-i-j}{2}!\frac{i-j}{2}!$ gives: 
\[j! \frac{n-i-j}{2}!\frac{i-j}{2}! \geq \left(\frac{j}{e}\right)^j\left(\frac{n-i-j}{2e}\right)^{\frac{n-i-j}{2e}} \left(\frac{i-j}{2e}\right)^{\frac{i-j}{2}} \sqrt{(2)^3j\frac{n-i-j}{2}\frac{i-j}{2}}\]
Separate this into two pieces \[j^j\left(\frac{n-i-j}{2}\right)^{\frac{n-i-j}{2}} \left(\frac{i-j}{2}\right)^{\frac{i-j}{2}}\] and \begin{align}\label{forgotten} e^{-n/2}\sqrt{(2)^3j\frac{n-i-j}{2}\frac{i-j}{2}}\end{align} Consider the maximal $j$ for the first piece with the $(2p)^j$ added.

\begin{align*} 
&\frac{d}{dj}\left((2p)^{-j}j^j\left(\frac{n-i-j}{2}\right)^{\frac{n-i-j}{2}}\left(\frac{i-j}{2}\right)^{\frac{i-j}{2}} \right)\\
&= \frac{d}{dj}e^{-j\log(2p) +j\log(j) + \frac{n-i-j}{2}\log\left(\frac{n-i-j}{2}\right) + \frac{i-j}{2}\log{\left(\frac{i-j}{2}\right)}} \\
&= \left(-\log(2p) + \log(j) +1 - \frac{1}{2}\log\left(\frac{n-i-j}{2}\right) -\frac{1}{2}- \frac{1}{2}\log\left(\frac{i-j}{2} \right) -\frac{1}{2} \right)\\
&e^{-j\log(2p) +j\log(j) + \frac{n-i-j}{2}\log\left(\frac{n-i-j}{2}\right) + \frac{i-j}{2}\log{\left(\frac{i-j}{2}\right)}} \\
&= \frac{1}{2}\log{\left(\frac{4j^2}{4p^2(n-i-j)(i-j)}\right)} e^{-j\log(2p) +j\log(j) + \frac{n-i-j}{2}\log\left(\frac{n-i-j}{2}\right) + \frac{i-j}{2}\log{\left(\frac{i-j}{2}\right)}}\\
\end{align*}

Solving for the value of $j$, $j'$, that gives $\frac{4j^2}{4p^2(n-i-j)(i-j)}=1$ gives:
\[ j' = \frac{\sqrt{1 + \frac{1-p^2}{p^2} \frac{4(n-i)(i)}{n^2}} -1}{2} \frac{p^2}{1-p^2}n = \frac{i(n-i)}{n}\frac{2}{1+\alpha} \]
Plugging this back into the expression $(2p)^{-j}j^j\left(\frac{n-i-j}{2}\right)^{\frac{n-i-j}{2}} \left(\frac{i-j}{2}\right)^{\frac{i-j}{2}}$ gives: 

\begin{align*} &e^{-j'\log(2p) +j'\log(j') + \frac{i-j'}{2}\log\left(\frac{j'-i}{2}\right) + \frac{n-i-j'}{2}\log\left(\frac{n-i-j'}{2}\right)} \\
=& e^{\frac{j'}{2}\log\left(\frac{4j'^2}{(2p)^2(j'-i)(n-i-j')}\right) + \frac{i}{2}\log\left(\frac{i-j'}{2}\right) + \frac{n-i}{2}\log\left(\frac{n-i-j'}{2}\right)} \end{align*}

Where $j'$ was chosen exactly to make the expression $\frac{4j^2}{(2p)^2(j-i)(n-i-j)}=1$. Further $i-j' = i\left(1 - \frac{n-i}{n}\frac{2}{1+ \alpha}\right)$ and $n-i-j' = (n-i)\left(1 - \frac{i}{n}\frac{2}{1 + \alpha}\right)$, so 

\[\frac{i}{2}\log\left(\frac{i-j'}{2} \right)=\frac{i}{2}\log\left(\frac{i}{2}\right) - \frac{i}{2}\log\left(1 - \frac{n-i}{n}\frac{2}{1+ \alpha}\right) \]
\[\frac{n-i}{2}\log\left(\frac{n-i-j'}{2}\right)= \frac{n-i}{2}\log\left(\frac{n-i}{2}\right) - \frac{n-i}{2}\log\left(1 - \frac{i}{n}\frac{2}{1 + \alpha}\right) \]

So this gives after another application of Stirling's formula, this time the upper bound from (\ref{Stirling}), a $\frac{i}{2}!\frac{n-i}{2}!\frac{e^{n/2}}{e^2}\sqrt{i(n-i)}$ plus an additional $\left(1 - \frac{n-i}{n}\frac{2}{1+ \alpha}\right)^{\frac{i}{2}}\left(1 - \frac{i}{n}\frac{2}{1 + \alpha}\right)^{\frac{n-i}{2}}$.

Attaching this back to the long neglected (\ref{forgotten}) gives:
\begin{align}&j! \left(\frac{n-i-j}{2}\right)!\left(\frac{i-j}{2}\right)! \\
 \geq& \label{repbou} \left(\frac{i}{2}\right)!\left(\frac{n-i}{2}\right)!\frac{2^{3/2}}{e^2}\sqrt{\frac{(n-i-j)(i-j)j}{(n-i)i}}\left(1 - \frac{n-i}{n}\frac{2}{1+ \alpha}\right)^{\frac{i}{2}}\left(1 - \frac{i}{n}\frac{2}{1 + \alpha}\right)^{\frac{n-i}{2}}\end{align}


The expression $\sqrt{\frac{(n-i-j)(i-j)j}{(n-i)i}}$ is minimized over cases where it is non-zero when $j=i-1$ where it is still at least $\frac{n-2i+1}{n-i+1}$. The expression is $0$ when one of $j$ or $i-j$ was $0$. This problem occurs because in these cases, the use of Stirling's approximation that gave a $0$ term was not needed. The bound will be adjusted to be non-zero and hold in all cases. 

When $j=0$, leaving out the $j!$ term during the application of Stirling's formula drops a $\sqrt{2j}$ so the square root of the fraction becomes $\sqrt{\frac{(n-i-0)(i-0)}{2(n-i)i} }=\sqrt{\frac{1}{2}}$. Suppose instead $j=i=n-i=n/2$, then two applications of Stirling's should be dropped. This drops a $\sqrt{2^2}$ but gains in back in the lack of cancellation of the denominators in $\sqrt{\frac{1}{(i/2)(n-i)/2}}$.  This gives the new fraction of $\sqrt{\frac{j}{i(n-i)}} = \frac{2}{n}$. Lastly, if $j=i\neq n-i$, one less use of Stirling's means a $2$ is dropped but regained from the denominator of $\sqrt{\frac{1}{(i/2)}}$ and the fraction is at worst $\sqrt{\frac{n-2i}{(n-i)}}$. The largest of these values when $i \neq n-i$ is $\sqrt{\frac{n-2i}{(n-i)}}$. When $n-i=i$, the largest is $\sqrt{\frac{2}{n}}$. 

The expression actually of interest is $(2p)^j{n/2 \choose j,\frac{i-j}{2},\frac{n-i-j}{2}}$. This proof gave lower bounds on its reciprocal without the $(n/2)!$.  This in turn gives an upper bound on the original expression. Taking the reciprocal of (\ref{repbou}), correcting $\sqrt{\frac{(n-i-j)(i-j)j}{(n-i)i}}$ as in the above paragraph to $\sqrt{\frac{n-2i}{n-i}}$, and adding a $\frac{i+1}{2}$ the number of terms in the sum, gives for $i \neq n/2$, \[ \sum_{j} (2p)^j {{n/2 \choose j, \frac{n-i-j}{2},\frac{i-j}{2}}} \leq {n/2 \choose i/2}\frac{i+1}{2}\frac{e^2}{2^{3/2}}\sqrt{\frac{n-i}{n-2i}}\left(1 - \frac{n-i}{n}\frac{2}{1+ \alpha}\right)^{-\frac{i}{2}}\left(1 - \frac{i}{n}\frac{2}{1 + \alpha}\right)^{-\frac{n-i}{2}} \]

The other case is $i = n/2$. Taking the reciprocal of (\ref{repbou}), correcting $\sqrt{\frac{(n-i-j)(i-j)j}{(n-i)i}}$ to $\sqrt{\frac{2}{n}}$, and adding a $\frac{i+1}{2}$ the number of terms in the sum, gives for $i=n/2$:
\[ \sum_{j} (2p)^j {{n/2 \choose j, \frac{n-i-j}{2},\frac{i-j}{2}}} \leq {n/2 \choose i/2}\frac{i+1}{2}\sqrt{n/2}\frac{e^2}{2^{3/2}}\left(1 - \frac{n-i}{n}\frac{2}{1+ \alpha}\right)^{-\frac{i}{2}}\left(1 - \frac{i}{n}\frac{2}{1 + \alpha}\right)^{-\frac{n-i}{2}} \]

\end{proof}

This gives for $\lambda$ with $\lambda_1=n-i> n/2$:

\[ \psi_{\lambda} \leq \frac{{n/2 \choose i/2}}{{n \choose i}}\left(\frac{n-i+1}{n-2i+1}\sqrt{\frac{n-i}{n-2i}}\frac{i+1}{2}\frac{e^2}{2^{3/2}}\right) \left(1 - \frac{n-i}{n}\frac{2}{1+ \alpha}\right)^{-\frac{i+1}{2}}\left(1 - \frac{i}{n}\frac{2}{1 + \alpha}\right)^{-\frac{n-i}{2}} \]

And for $\lambda_1 \leq n/2$,
\[ \psi_{\lambda} \leq \frac{{n/2 \choose n/4}}{{n \choose n/2}}\left((n/2+1)^2/2\sqrt{n/2}\frac{e^2}{2^{3/2}}\right) \left(1 - \frac{1}{2}\frac{2}{1+ \alpha}\right)^{-\frac{n}{2}} \]

The next proposition brings things together into one expression:
\begin{proposition}\label{finalpsi}
 For $\lambda$ with $\lambda_1=n-i$, $i < n/2$:
\[\psi_\lambda \leq e^{-i \log(\frac{2}{1+p}) + \log\left(\frac{e^2(i+1)}{2^{5/2}}\left(\frac{n-i}{n-2i}\right)^{3/2} \right)}\]

For $\lambda_1 \frac{n}{2}$,
\[\psi_{\lambda} \leq e^{-\frac{n}{2}\log{\frac{2}{1+p}} + \log(\frac{n^{3/2}(n+2)e^2}{8})}\]
\end{proposition}
\begin{proof}

The bound on $\psi_{[n/2,n/2]}$ simplifies considerably. In particular, $\alpha = \frac{1}{p}$, giving,
\begin{align*}
\psi_{[n/2,n/2]} & \leq e^{-n/2\log\left(1 + \frac{1 - p^2}{p^2}\left(\frac{1}{1+ \alpha^2}\right)^2\right) + \log\left(\frac{n^{3/2}(n+2)e^2}{8}\right)}\\
&\leq e^{-n/2\log\left(\frac{2}{1+p}\right) + \log\left(\frac{n^{3/2}(n+2)e^2}{8}\right)}
\end{align*}

For $i < \frac{n}{2}$, $\alpha \leq \frac{1}{p}$
\begin{align*} \psi_{\lambda} \leq& e^{-\frac{i}{2}\log\left(1 + \left(\frac{1-p^2}{p^2}\right)\left(\frac{2(n-i)}{n}\right)^2\left(\frac{1}{1+ \alpha}\right)^2\right) - \frac{n-i}{2}\log\left(1 + \left(\frac{1-p^2}{p^2}\right)\left(\frac{2i}{n}\right)^2\left(\frac{1}{1+ \alpha}\right)^2\right) +\log{\left(\frac{e^2(i+1)}{2^{5/2}}\frac{n-i+1}{n-2i+1}\sqrt{\frac{n-i}{n-2i}}\right)}} \end{align*}
First, $\alpha \leq \frac{1}{p}$. Secondly, $\log\left(1 + \frac{1-p}{1+p}\frac{2(n-i)}{n}\right) + \frac{n-i}{i}\log\left(1 + \frac{1-p}{1+p}\frac{2i}{n}\right)$ can be seen to be decreasing with $i$ (by differentiation with respect to $i$). Therefore it can be bounded from below by $2\log\left(\frac{2}{1+p}\right)$. So,
\begin{align*}
\psi_{\lambda}\leq& e^{-\frac{i}{2}\left(\log\left(1 + \left(\frac{1-p}{1+p}\right)\left(\frac{2(n-i)}{n}\right)^2\right) - \frac{n-i}{i}\log\left(1 + \left(\frac{1-p}{1+p}\right)\left(\frac{2i}{n}\right)^2\right)^2\right) + \log\left(\frac{e^2(i+1)}{2^{5/2}} \left(\frac{n-i}{n-2i}\right)^{3/2}\right)}\\
\leq & e^{-i \log\left(\frac{2}{1+p}\right) + \log\left(\frac{e^2(i+1)}{2^{5/2}} \left(\frac{n-i}{n-2i}\right)^{3/2}\right)} \\
\end{align*}
\end{proof}

\begin{proposition}\label{invismall}
For $i \leq p\sqrt{n-2i+2}$, \[\sum_{j \leq i, i-j \text{even}} (2p)^j {{n/2 \choose j, \frac{n-i-j}{2},\frac{i-j}{2}}}  \leq \frac{1}{1 - \frac{i(i-1)}{2p^2(n-2i+2)}}(2p)^i {{n/2 \choose i}}\] 
\end{proposition}
\begin{proof}
This follows from the $j=i$ term being larger than the $j=i-2$ term by a factor of $2$ under the condition on $i$, and the terms with $j$ smaller continue to fall off even faster. 

\[ \frac{ (2p)^j {{ n/2 \choose j, \frac{i-j}{2}, \frac{n-i-j}{2}}}}{(2p)^{j-2}{{ n/2 \choose j-2, \frac{n-i-j+2}{2}, \frac{i-j+2}{2}}}} = p^2 \frac{(i-j+2)(n-i-j+2)}{j(j-1)} \]

When $i = j \leq  p\sqrt{n-2i+2}$, this is at least $\frac{2p^2(n-2i+2)}{i(i-2)} \geq 2$. As $j$ decreases, the numerator increases and the denominator decreases, so the terms are falling off exponentially faster. 

\end{proof}

\begin{proposition}\label{neg1}
\[ \psi_{1^n} = (2p-1)^{n/2} \]
\end{proposition}
\begin{proof}
By Murnaghan-Nakayama, $\chi_{1^n}(1^{n-2s},2^{s}) = (-1)^s$ since all the $2$-cycles insert vertically and these are exactly insertions covering an even number of rows.
\[ \psi_{1^n} = \sum_{s=0}^{n/2} p^{n/2-s}(1-p)^s {{n/2 \choose s}} \frac{ \chi_{1^n}(1^{n-2s},2^{s})}{d_{1^n}} = \sum_{s=0}^{n/2} {n/2 \choose s}p^{n/2-s}(1-p)^s (-1)^s = (2p-1)^{n/2} \]
\end{proof}

And with this bound on $\psi_\lambda$, the upper bound lemma is at hand.

\begin{theorem}\label{theorem:invup}
For $t = \log_{\frac{2}{1+p}}(n) + \frac{c}{\log(\frac{2}{1+p})}$, $n$ such that $\frac{10\log(n+2)}{\sqrt{(n+2)/2}-1} \leq \log\left(\frac{2}{1+p}\right)$ and $n-1 > \sqrt{n/2}(1 + \log(n))$, then 
\[||P^{*t} - U||_{TV} \leq  e^{-c/2}\] 
\end{theorem}
\begin{proof}
For $\lambda_1 \geq \lambda_1'$, the bound on $\psi_\lambda$ is valid for $\psi_{\lambda'}$ since $|\chi_\lambda| = |\chi_{\lambda'}|$ and sign was ignored in the bounds for all but $\lambda = 1^n$. So this leaves the cases for $\lambda \neq [n]$, $[1^n]$, $\lambda_1 > n/2$ or $\lambda_1' >n/2$, and $\lambda_1'\leq \lambda_1 \leq n/2$ or $\lambda_1 \leq \lambda_1' \leq n/2$. These are treated in turn.
\[ \sum_{\lambda \neq 1} d_\lambda^2 \left(\psi_{\lambda}\right)^{2t} \leq \psi_{1^n}^{2t} + 2\sum_{i=1}^{n/2-1} \psi_{[n-i,i]}^{2t}\left(\sum_{\lambda: \lambda_1 =n-i} d_\lambda^2\right)  + 2\left(\psi_{[n/2,n/2]}\right)^{2t}\sum_{\lambda: \lambda_1' \leq \lambda_1 \leq \frac{n}{2}} d_\lambda^2  \]
And $\sum_{\lambda: \lambda_1 =n-i} d_\lambda^2 \leq {{n \choose i}}^2i! \leq e^{i\log(n)}\frac{1}{i!}$, $\sum_{\lambda: \lambda_1 \leq n/2} d_\lambda^2 \leq \sum_{\lambda} d_{\lambda}^2 = n!$. The case of $i<n/2$ further breaks into the case of $i \leq p\sqrt{(n+2)/2}-1$ and $i \geq p\sqrt{(n+2)/2}$. This reduces things to:

\begin{align}\label{pform} \psi_{1^n}^{2t} + 2\sum_{i=1}^{n/2-1} \frac{(n)_i^2}{i!}\psi_{[n-i,i]}^{2t} + n!\left(\psi_{[n/2,n/2]}\right)^{2t} \end{align}

\noindent{Case 1: $[1^n]$}

\begin{align} (2p-1)^{2\frac{\log(n) + c}{\log\left(\frac{2}{1+p}\right)}} &= e^{\log(2p-1)\frac{\log(n)+c}{\log\left(\frac{2}{1+p}\right)}} \nonumber \\
&= e^{-2\left(\log(n)+c\right)\frac{\log\left(\frac{1}{2p-1}\right)}{\log\left(\frac{2}{1+p}\right)}} \nonumber\\
\label{inv1n}&\leq e^{-2\log(n)}e^{-2c} \\
&= \frac{1}{n^2}e^{-2c} \nonumber \end{align}
The inequality (\ref{inv1n}) follows since $\frac{2}{1+p} \leq \frac{1}{2p-1}$ for $\frac{1}{2}\leq p <1$. 

In Proposition \ref{exacteigen}, it is found that $\psi_{[n-1,1]} = p - (1-p)\frac{1}{n-1} \leq \frac{2}{1+p}$. This means the term in (\ref{pform}) for $i=1$ is at most, 
\begin{align*}
(n-1)^2\left(p - (1-p)\frac{1}{n-1}\right)^{2\frac{\log(n) + c}{\log\left(\frac{2}{1+p}\right)}} & = (n-1)^2e^{-2(\log(n) + c) \frac{\log(p - (1-p)\frac{1}{n-1})}{\log\left(\frac{1+p}{2}\right)}}\\ 
&\leq e^{-2c}
\end{align*}

\noindent{Case 2: $i \leq p\sqrt{(n+2)/2}-1$}

For $i \leq p\sqrt{(n+2)/2}-1 \leq p\sqrt{(n-2i+2)/2}$, assume that $n$ is sufficiently large that \[\frac{2}{n-2i+1} \leq p^2\left(\log\left(\frac{2}{1+p}\right)\right)^2\] This assumption on $n,p$ is weaker than the assumption made in the next two cases.

Note that $\frac{\log\left(\frac{1}{p}\right)}{\log\left(\frac{2}{1+p}\right)} \geq 2$.
\begin{align}
\label{ismall1}e^{2i\log(n)}\frac{1}{i!} \psi_{[n-i,i]}^{2t} \leq& \frac{e^{2i\log(n)}}{i!} \left( \frac{(2p)^i \frac{1}{1- \frac{i(i-1)}{p^2(n-2i+2)}} {n/2 \choose i}}{{n \choose i}\frac{n-2i+1}{n-i+1}} \right)^{2t} \\
=& \frac{e^{2i\log(n)}}{i!} \left( \frac{(2p)^i(n/2)_i}{(n)_i} \frac{n-i+1}{n-2i+1}\frac{1}{1- \frac{i(i-1)}{p^2(n-2i+2)}} \right)^{2t} \nonumber\\
\label{ismall3}\leq&  \frac{e^{2i\log(n)}}{i!} \left(p^i\frac{n-2i+2}{n-2i+1}\frac{1}{1- \frac{i(i-1)}{p^2(n-2i+2)}} \right)^{2\frac{\log(n) + c}{\log\left(\frac{2}{1+p}\right)}} \\
=& \frac{e^{2i\log(n)}}{i!} e^{-2i(\log(n) + c)\frac{\log\left(\frac{1}{p}\right)}{\log\left(\frac{2}{1+p}\right)}}\left(\frac{n-2i+2}{n-2i+1}\frac{1}{1- \frac{i(i-1)}{p^2(n-2i+2)}}\right)^{2\frac{\log(n) + c}{\log\left(\frac{2}{1+p}\right)}} \nonumber\\
\label{ismall5}\leq& \frac{e^{-2ic}}{i!} e^{-2i(\log(n) + c)}e^{\left(\frac{1}{n-2i+1} + \frac{2i(i-1)}{p^2(n-2i+2)}\right)2\frac{\log(n) + c}{\log\left(\frac{2}{1+p}\right)}} \\
\label{ismall6}\leq& \frac{e^{-2ic}}{i!} e^{-2i(\log(n) + c)}e^{\left( \frac{2i^2}{p^2(n-2i+1)}\right)2\frac{\log(n) + c}{\log\left(\frac{2}{1+p}\right)}} \\
\label{ismall7}\leq& \frac{e^{-2ic}}{i!} e^{-2i(\log(n) + c)}e^{2i(\log(n) + c)} \\
=& \frac{e^{-2ic}}{i!} \\
\end{align}
The first inequality (\ref{ismall1}) follows by Proposition \ref{invismall}. In (\ref{ismall3}), \[\frac{(n/2)_i}{(n)_i}\frac{n-i+1}{n-2i+1} = \frac{(n/2) \cdots (n/2-i+1) (n-i+1)}{n(n-1) \cdots (n-i+1) (n-2i+1)} \leq 2^{-i}\frac{n-2i+2}{n-2i+1}\] The next inequality (\ref{ismall5}) consists of three parts. The first is that $\frac{\log\left(\frac{1}{p}\right)}{\log\left(\frac{2}{1+p}\right)} \geq 2$. Next, $\frac{n-2i+2}{n-2i+1} = e^{\log\left(1 + \frac{1}{n-2i+1}\right)} \leq e^{\frac{1}{n-2i+1}}$. The third uses that $i^2 \leq p^2/2(n-2i+2)$, the condition on $i$, to conclude $\frac{1}{1 - \frac{i(i-1)}{p^2(n-2i+2)}} = 1 + \frac{i(i-1)}{p^2(n-2i+1)}\frac{1}{1-\frac{i(i-1)}{p^2(n-2i+2)}} \leq 1 + \frac{2i(i-1)}{p^2(n-2i+1)}\leq e^{\frac{2i(i-1)}{p^2(n-2i+1)}}$. The inequality in (\ref{ismall6}) consists of $2i \leq p^2$ and $n-2i+1 \leq n-2i+2$. Finally, to get (\ref{ismall7}), it suffices to show $\frac{2i}{p^2(n-2i+1)\log\left(\frac{2}{1+p}\right)} \leq 1$. Since $i \leq p \sqrt{\frac{n-2i+2}{2}}$, this simplifies to $\frac{\sqrt{2}}{\sqrt{n-2i+1}}\leq \log\left(\frac{2}{1+p}\right)$. This was the assumption on $n,p$ made at the beginning of this case.

\noindent{Case 3: $i \geq \sqrt{(n+2)/2}$}

For $i \geq p\sqrt{(n+2)/2}$, assume $n$ sufficiently large that $\frac{10\log(n)}{\sqrt{(n+2)/2}-1} \leq \log\left(\frac{2}{1+p}\right)$
\begin{align}
\label{bigi1}e^{2i\log(n)}\frac{1}{i!} \psi_{[n-i,i]}^{2t} \leq& \frac{e^{2i\log(n)}}{i!}e^{2\frac{\log(n) + c}{\log(\frac{2}{1+p})}\left(-i \log\left(\frac{2}{1+p}\right) + \log\left(\frac{i+1}{2}\left(\frac{n-i}{n-2i}\right)^{3/2}\frac{e^2}{2^{3/2}} \right)\right)} \\
=& \frac{e^{-2ci}}{i!}e^{2(\log(n)+c)\left(\frac{\log\left(\frac{i+1}{2}\left(1 + \frac{i}{2}\right)^{3/2}\frac{e^2}{2^{3/2}} \right)}{\log\left(\frac{2}{1+p}\right)}\right)} \nonumber\\
\label{bigi3}\leq& \frac{e^{-2ci}}{i!}e^{2(\log(n)+c)\log\left(\frac{5}{2}\left(\frac{i+2}{2}\right) +1\right)\frac{i}{10\log(n)}}\\
\label{bigi4}\leq& \frac{e^{-2ci}}{i!}e^{i\log\left(\frac{i}{e}\right) + ci} \\
\label{bigi5}\leq& e^{-ic}
\end{align}
The bound on $\psi_{[n-i,i]}$ from Proposition \ref{finalpsi} gives (\ref{bigi1}). The assumption about $n$ and $p$ gives that $\frac{1}{\log\left(\frac{2}{1+p}\right)} \leq \frac{i}{10\log(n)}$, (\ref{bigi3}) follows. At this step, $\log\left(\frac{i+1}{2}\left(1 + \frac{i}{2}\right)^{3}{2}\right) \leq \frac{5}{2}\log\left(\frac{i+2}{2}\right) $ was also used. To arrive at (\ref{bigi4}) take the following steps. Multiplied out the expression in (\ref{bigi4}) becomes $\frac{1}{2}i\log(\frac{i+2}{2}) + \frac{1}{2}ic\log(\frac{i+2}{2})/\log(n) + \frac{1}{5}i + \frac{1}{5}ci/\log(n)$. Then $\frac{1}{2}i\log(\frac{i+2}{2}) \leq i\log(i/e) + i$ and $i + \frac{1}{2}ic + \frac{1}{5}i + \frac{2}{5}{ic} \leq ic$. Finally, taking that $i! \geq (i/e)^i$ gives (\ref{bigi5}).

\noindent{Case 4: $\lambda_1 \leq n/2$}

As was found above, the bound on $\psi_{[n/2,n/2]}$ simplifies considerably. Assume $n$ sufficiently large that $\frac{10\log(n+2)}{\sqrt{(n+2)/2}-1} \leq \log\left(\frac{2}{1+p}\right)$ and that $n-1 \geq \sqrt{n/2} + \log(n)\sqrt{n/2}$.
\begin{align}
\label{invhalf1}n! \psi_{[n/2,n/2]}^{2t} \leq&  n! \left(e^{-n/2\log{\frac{2}{1+p}} + \log(\frac{n^{3/2}(n+2)e^2}{8})} \right)^{2\frac{\log(n)+c}{\log(\frac{2}{1+p})}} \\
\label{invhalf2}\leq& e^{n\log(n)}e^{-n(\log(n) + c)\left(\frac{(\log(\frac{2}{1+p})- \frac{2}{n}\log(\frac{n^{3/2}(n+2)e^2}{8}))}{\log(\frac{2}{1+p})}\right)} \\
\label{invhalf3}\leq&  e^{-nc + (\log(n)+c)(5\log(n+2))/\log\left(\frac{2}{1+p}\right)} \\
\label{invhalf4}\leq& e^{-c(n - \sqrt{n/2}) + \log(n)\sqrt{n/2}} \\
\label{invhalf5} \leq & e^{-c}
\end{align}
From the bound on $\psi_{[n/2,n/2]}$ in Proposition \ref{finalpsi}, (\ref{invhalf1}) follows. The next inequality follows from taking $n! \leq e^{n\log(n)}$. That $e^2 \leq 8$ gives (\ref{invhalf3}). Using the assumption that $\frac{10\log(n+2)}{\sqrt{(n+2)/2}-1} \leq \log\left(\frac{2}{1+p}\right)$ and multiplying out the terms gives (\ref{invhalf4}). Finally, with the assumption on $n$, (\ref{invhalf5}) follows.

When the above requirements are met, by the upper bound formula, 
\begin{align*} &||P^{*t} - U||_{TV}^2 \\
\leq& \frac{1}{4} \left(\frac{1}{n^2}e^{-2c} + e^{-2c} + \left(\sum_{i=2}^{ p\sqrt{(n+2)/2}-1} \frac{e^{-2ic}}{i!}\right) + \sum_{i = p\sqrt{(n+2)/2}-1}^{n/2-1} e^{-ic} +  e^{-c} \right)\\
\leq & \frac{1}{4}\left(\frac{1}{n^2}e^{-2c} + \frac{e^{-c}}{1 + e^{-c}} + e^{-c}\right)\\
\leq& e^{-c}\\
\end{align*}

\end{proof}

\section{Lower Bound on Mixing}\label{lowerbound}

The representation slowest to vanish for this walk is $[n-1,1]$, so its character gives a random variable where $P^{*t}(\cdot)$ and $\pi(\cdot)$ differ significantly. Using a lower bound formula similar to Chebychev's inequality after calculating the first and second moments of this character will give a lower bound on mixing of $\log_{\frac{1}{p}}(n)$.

\begin{proposition}\cite{LPW}\label{PeresL}
For $\gamma$, $\nu$ two probability distributions on $\Omega$, and $f$ a real valued function on $\Omega$, if \[|E_\gamma(f) - E_\nu(f)| \geq r\sigma \] where $\sigma^2=[Var_\gamma(f) + Var_\nu(f)]/2$, then \[||\gamma - \nu ||_{TV} \geq 1 - \frac{4}{4+ r^2} \]
\end{proposition}

In this case, $\nu=U$ is the stationary distribution of the walk, uniform over all permutations. As seen in \cite{Diaconis}, \[E_U\left(\chi_{[n-1,1]}\right) = 0, Var_U(\chi_{n-1,1}) = 1 \] These follow for any non-trivial characters by basic tenets of representation theory. For the first, by orthogonality of characters, $\sum_{g\in G} \chi_{\lambda}(g) = 0$. For the second, $\sum_{g \in G} \chi_\lambda{g}^2 = |G|$.

\begin{proposition}\label{exacteigen}

\[E_{P^{*t}}\chi_{[n]} = 1\]
\[E_{P^{*t}}\left(\chi_{[n-1,1]}\right) = (n-1)\left(p - (1-p)\frac{1}{n-1}\right)^t\]

\[E_{P^{*t}}\chi_{[n-2,2]} = \frac{n(n-3)}{2}\left(p^2 - \frac{(1-p)^2}{n-3} \right)^t\]
\[E_{P^{*t}}\chi_{[n-2,1,1]} = \frac{(n-1)(n-2)}{2}\left(p^2 - \frac{1-p^2}{n-1} - \frac{2}{(n-1)(n-2)} \right)^t\]

\end{proposition}
\begin{proof}

For an irreducible representation $\lambda$, since $P$ is a class function, by Schur's Lemma, the Fourier transform of $P$ is a constant $\psi_\lambda$ times the identity matrix. \[ \hat{P}(\lambda) = \psi_\lambda I_{d_\lambda}\]Moreover, $\hat{P^{*t}}(\lambda) = \left(\hat{P}(\lambda)\right)^t$. This leads to the following formula for the expected value of a character over the walk:

\[ \mathbb{E}_{P^{*t}} (\chi_\lambda) = \sum P^{*t}(g)\tr(\lambda(g)) = \tr(\sum P^{*t}(g)\lambda(g)) = \tr \hat{P^{*t}}(\lambda) = d_\lambda \psi_\lambda^t\]

The method of choice to compute the expectation for $\chi_{[n-1,1]}$ will be to directly compute $\psi_{\lambda}$. Recall, \[\psi_\lambda = \sum_{s=0}^{n/2} p^{n/2-s}(1-p)^s {{n/2 \choose s}} \frac{ \chi_{\lambda}(1^{n-2s},2^{s})}{d_{\lambda}}\]


Further, the character polynomials of the representation gives that:
\[ \chi_{[n]}(1^{n-2s},2^s) = 1, \chi_{[n-1,1]}(1^{n-2s},2^s) = n-2s-1\] \[ \chi_{[n-2,2]}(1^{n-2s},2^s) = { n-2s \choose 2 } - (n-2s) +s\] \[ \chi_{[n-2,1,1]}(1^{n-2s},2^s) = {n-2s \choose 2} - (n-2s) - s \]
\[\chi_{[n-1,1]}(1^{n-2s},2^s) = n-2s-1\]
Then, using that $\sum_{s=0}^{n/2} p^{n/2-s}(1-p)^s {{n/2 \choose s}}(s)_k = (1-p)^k(\frac{n}{2})_k$, 
\[\psi_{[n]} = \sum_{s=0}^{n/2} p^{n/2-s}(1-p)^s {{n/2 \choose s}} 1 =1\]

\begin{align*} \psi_{[n-1,1]} &= \sum_{s=0}^{n/2} p^{n/2-s}(1-p)^s {{n/2 \choose s}} \frac{ \chi_{[n-1,1]}(1^{n-2s},2^{s})}{d_{[n-1,1]}}\\
&= \sum_{s=0}^{n/2} p^{n/2-s}(1-p)^s {{n/2 \choose s}}\frac{n-2s-1}{n-1} \\
& =1 - \frac{2}{n-1}\sum_{s=0}^{n/2} p^{n/2-s}(1-p)^s {{n/2 \choose s}}s \\
&= 1 - \frac{2}{n-1}(1-p)\frac{n}{2} \\
& = p - (1-p)\frac{1}{n-1} \\
\end{align*}
%

\begin{align*} \psi_{[n-2,2]} &= \sum_{s=0}^{n/2} p^{n/2-s}(1-p)^s {{n/2 \choose s}} \frac{ \chi_{[n-2,2]}(1^{n-2s},2^{s})}{d_{[n-2,2]}}\\
&= \sum_{s=0}^{n/2} p^{n/2-s}(1-p)^s {{n/2 \choose s}}\frac{{ n-2s \choose 2 } - (n-2s) +s}{\frac{n(n-3)}{2}} \\
&= \sum_{s=0}^{n/2} p^{n/2-s}(1-p)^s {{n/2 \choose s}}\frac{n(n-3) + 4(s)_2 -4(n-3)s}{n(n-3)} \\
&= 1 + \frac{4}{n(n-3)}(1-p)^2\left(\frac{n}{2}\right)_2 - \frac{4(n-3)}{n(n-3)}(1-p)\frac{n}{2} \\
& = p^2 - \frac{1}{n-3}(1-p)^2 
\end{align*}
\begin{align*} \psi_{[n-2,1,1]} &= \sum_{s=0}^{n/2} p^{n/2-s}(1-p)^s {{n/2 \choose s}} \frac{ \chi_{[n-2,1,1]}(1^{n-2s},2^{s})}{d_{[n-2,1,1]}}\\
&=  \sum_{s=0}^{n/2} p^{n/2-s}(1-p)^s {{n/2 \choose s}}\frac{{n-2s \choose 2} - (n-2s) - s+1}{\frac{(n-1)(n-2)}{2}} \\
& = \sum_{s=0}^{n/2} p^{n/2-s}(1-p)^s {{n/2 \choose s}}\frac{(n-1)(n-2) +4(s)_2 -4(n-2)s}{(n-1)(n-2)} \\
& = 1 + \frac{4}{(n-1)(n-2)}(1-p)^2\left(\frac{n}{2}\right)\left(\frac{n-2}{2}\right)  - \frac{4(n-2)}{(n-1)(n-2)}(1-p)\frac{n}{2} \\
& = p^2 - \frac{1-p^2}{n-1} 
\end{align*}

\end{proof}

This gives: \begin{align}\sigma^2 &= \frac{1}{2}\big(1 + 1 + (n-1)\left(p - (1-p)\frac{1}{n-1}\right)^{t} + \frac{n(n-3)}{2}\left(p^2 - \frac{1}{n-3}(1-p)^2  \right)^t \\
&+ \frac{(n-1)(n-2)}{2}\left(p^2 - \frac{1-p^2}{n-1} - \frac{2}{(n-1)(n-2)} \right)^t - (n-1)^2\left(p - (1-p)\frac{1}{n-1}\right)^{2t}\big) \end{align}

\begin{theorem}\label{theorem:invlb}
For $t < \log_{\frac{1}{p}}(n-1) - \frac{c}{\log(\frac{1}{p})}$ with  $c \leq \frac{1}{2}\left( \log(n) - \log\log(n) + \log\left(2\frac{1-p}{p}\right)\right)$, $ ||P^{*t} - U ||_{TV} \geq 1 - \frac{1}{1 + A^2e^{2c} - \frac{2}{4 + Ae^c}}$ where $A = \left(1 - \left(\frac{1}{p}-1\right)\frac{1}{n-1}\right)^{t}$. Note that for $p \geq \frac{1}{2}$, $1 - \frac{\log(n)}{n} \leq A \leq 1$.

\end{theorem}
\begin{proof}

From Proposition \ref{PeresL} $||P^{*t} - U ||_{TV} \geq 1 - \frac{4}{4 + r^2}$ for \[r \leq \frac{(n-1)\left(p - (1-p)\frac{1}{n-1}\right)^t}{\sigma}\]

Let $ A =  \left(1 - \left(\frac{1}{p}-1\right)\frac{1}{n-1}\right)^{t}$,
\begin{align*}
(n-1)\left(p - (1-p)\frac{1}{n-1}\right)^t &= (n-1)p^t\left( 1 - \left(\frac{1}{p}-1\right)\frac{1}{n-1}\right)^t \\
&=e^c\left( 1 - \left(\frac{1}{p}-1\right)\frac{1}{n-1}\right)^t \\
&= e^c A\\
\end{align*}

Let $B = \left(1 - \left(\frac{1-p}{p}\right)^2\frac{1}{n-3}\right)^{t}$. By observation or the monotoniticity conditions from Lemma \ref{twopartbiggest}, $\psi_{[n-2,2]} \geq \psi_{n-2,1^2}$. Replacing $\psi_{[n-2,1^2]}$ by $\psi_{[n-2,2]}$ makes $\sigma$ larger.

\begin{align}
\sigma^2 &\label{lb1} \leq 1 + \frac{1}{2}Ae^C + \frac{1}{2}\left({n(n-3)}{2} + \frac{(n-1)(n-2)}{2}\right)\left(p^2 - \frac{1}{n-3}(1-p)^2  \right)^t - \frac{1}{2}A^2e^{2c} \\
& = 1 + \frac{1}{2}Ae^C + \frac{1}{2}(n^2 -3n +1)p^{2t}\left(1 - \left(\frac{1-p}{p}\right)^2\frac{1}{n-3} \right)^t - \frac{1}{2}A^2e^2c \nonumber\\
& = 1 + \frac{1}{2}Ae^C + \frac{B}{2}\left(1- \frac{n}{(n-1)^2}\right)e^{2c} - \frac{A^2}{2}e^{2c} \nonumber\\
\end{align}  
The inequality (\ref{lb1}) comes from using the above value of $(n-1)\psi_{[n-1,1]}^t = Ae^c$ and replacing  $\psi_{[n-2,1^2]}$ by $\psi_{[n-2,2]}$. To make the bound work, the terms with $e^{2c}$ need to have coefficient that is $o(1)$.

\begin{align}
B - A^2 & = \left(1 - \left(\frac{1-p}{p}\right)^2\frac{1}{n-3} \right)^t - \left( 1 - \frac{1-p}{p}\frac{1}{n-1}\right)^{2t} \nonumber \\
\label{lb2} &\leq \left(\left(1 - \left(\frac{1-p}{p}\right)^2\frac{1}{n-3} \right) -  \left( 1 - \frac{1-p}{p}\frac{1}{n-1}\right)^{2}\right)t\left(1 - \left(\frac{1-p}{p}\right)^2\frac{1}{n-3} \right)^{t-1} \\
& = \left(\frac{1-p}{p}\right)\left(\frac{2}{n-1} - \frac{1}{n-3} - \left(\frac{1-p}{p}\right)\frac{1}{(n-1)^2} \right)t\left(1 - \left(\frac{1-p}{p}\right)^2\frac{1}{n-3} \right)^{t-1}\nonumber \\
\label{lb3}& \leq \left(\frac{1-p}{p}\right)\frac{t}{n-1} 
\end{align}
The inequality (\ref{lb2}) follows using that $a^t - b^{2t} = (a-b^2)(a^{t-1} + a^{t-2}b + ... + b^{t-1} \leq (a-b^2)t\max{a,b^2}^t$. For all $p$, $\frac{2}{n-1} - \frac{1}{n-3} - \left(\frac{1-p}{p}\right)\frac{1}{(n-1)^2} \leq \frac{1}{n-1}$, so (\ref{lb3}) follows. 

This gives $\sigma^2 \leq 1 + \frac{1}{2}Ae^C + \frac{1}{2}e^{2c}\frac{\log(n-1)}{n-1}\frac{1-p}{p}$. For $c \leq \frac{1}{2}\left( \log(n) - \log\log(n) + \log\left(2\frac{1-p}{p}\right)\right)$ this least term is at most $1$. So for these values of $c$, the following value of $r^2$ is less than the needed bound:

\[r^2 = \frac{A^2e^{2c}}{2 + \frac{1}{2}Ae^c} = 2A^2e^{2c} - \frac{8}{4 + Ae^c}\]

This gives a lower bound of $ 1- \frac{4}{4 + r^2} = 1 - \frac{1}{1 + A^2e^{2c} - \frac{2}{4 + Ae^c}}$

For $t \leq \log(n)$, $p \geq \frac{1}{n}$, using that $1-xt \leq (1-x)^t$ when $xt \leq 1$, \[1 - \frac{1-p}{p}\frac{1}{\log(1/p)}\frac{\log(n)}{n} \leq A \leq 1 \]. In the case $p \geq \frac{1}{2}$ $A \geq 1 - \frac{\log(n)}{n}$ barely effects these bounds.

\end{proof}

For $p \geq \frac{1}{2}$ this gives a lower bound for mixing of $\frac{\log(n) -c}{\log(1/p)}$ which is off by just over a factor of two from the upper bound of $\frac{\log(n) + c}{\log(2/(1+p))}$. When $p$ is small, less than $\frac{1}{n- \sqrt{n}}$, $\psi_{[n-1,1]} = p - \frac{1-p}{n-1}$ is no longer the largest eigenvalue in magnitude as \[\lim_{n \rightarrow \infty} |\psi_{[1^n]}| = \lim_{n \rightarrow \infty} \left(1 - \frac{2}{n}  \right)^{n/2} = \frac{1}{e}\] 
This cross over happens around $\frac{1}{p} = W(e^n) \approx n - \log(n) + o(1)$ where $W$ is the product log function, also known as the Lambert $W$-function, as:

\[ \log((1- 2p)^{n/2} = n/2\log(1-2p) \approx pn \]

\[\log\left(p - \frac{1-p}{n-1}\right) = \log(p) + \log\left(1 - \frac{1-p}{pn}\right) \approx \log(p) - \frac{1-p}{pn}\]

For $p \approx \frac{1}{n}$ that last term contributes at most a constant, leaving the equation $pn \approx \log(p)$ with solution $p = W(e^n)$.

As discussed in the introduction, the conjectured mixing time for the walk of $\log_{1/p}(n)$ is only a conjecture for $p$ bounded away from $0$. As $p \rightarrow 0$, in the involution walk it becomes vanishingly unlikely that anything other than a perfect matching will be selected. This means the walk acts like the random walk generated by perfect matchings which at even steps is confined to the alternating group, but rapidly mixes on that set. This means it will take longer and longer to get a random parity in the involution walk, which is a prerequisite to be mixed. On the other hand, $\log_{1/p}(n) \rightarrow 0$ if $n$ is held constant and $p \rightarrow 0$, which is not compatible with the behavior of the walk. 

Since the issue here is parity rather than fixed points, the following bound follows from looking at the behavior of the random variable $\chi_{[1^n]}$ rather than $\chi_{[n-1,1]}$ in the first lower bound. Moreover, since this random variable takes on only two values, a lower bound on total variation can be found directly by evaluating on the set of permutations where $\chi_{[1^n]}$ is $1$, the even permutations. 

\begin{proposition}\label{tinyp}
For the involution walk with $t$ even, $||P^{*t} - U||_{TV} \geq \frac{1}{2}(1-2p)^{tn/2}$. When $t \leq \frac{1}{n^2p}$ with $p \leq \frac{1}{4}$, total variation distance at least $\frac{1}{2} - \frac{1}{n}$. 
\end{proposition}
\begin{proof}
Using the same facts about representations used to find the expected value of $\chi_{[n-1,1]}$ above,
\[ \mathbb{E}_{P^{*t}}(\chi_{[1^n]}) = (2p-1)^{tn/2}\]

We also know $\chi_{[1^n]}$ is $1$ on the even permutations $A_n \subset S_n$ and $-1$ on the odd permutations. This means \[\mathbb{E}_{P^{*t}}(\chi_{[1^n]}) = P^{*t}(A_n) - (1-P^{*t}(A_n) = (2p-1)^{tn/2}\]

The total variation distance by definition is $||P^{*t} - U||_{TV} =\sup_{A \subset S_n} \left| P^{*t}(A) - U(A)\right|$. Since we know the probability of the alternating group, this gives a lower bound on the total variation distance as $|\frac{1}{2}((2p-1)^{tn/2} +1) - \frac{1}{2}| = \frac{1}{2}(1-2p)^{tn/2}$.

If $t$ and $n$ are held constant, this gives as $p \rightarrow 0$, the total variation distance is at least $\frac{1}{2}$. For $t$ constant, and $p$ decreasing faster than $n$ so that $pn^2 \rightarrow 0$, the total variation distance will still be almost $\frac{1}{2}$ as when $p \leq \frac{1}{4}$ and  $t \leq \frac{1}{n^2p}$ :

\begin{align}
\label{lblte}\frac{1}{2}(1-2p)^{tn/2} & \geq e^{-2(2p)(tn/2)}\\
& = \frac{1}{2}e^{-2p t n} \nonumber\\
& \label{lbetl} \geq \frac{1}{2}(1-2p tn) \\
& \geq \frac{1}{2} - \frac{1}{n} \nonumber\\
\end{align}
Where \ref{lblte} follows from $1-x \geq e^{-2x}$ for $x \leq \frac{1}{2}$ and \ref{lbetl} from $e^{-x} \geq 1-x$ for all $x$. 

\end{proof}

\section{Conjectured Separation Distance Bound}\label{conjsep}

Separation distance is for a random walk on the symmetric group, $\sep(t) = \max_g(1 - n!P^{*t}(g))$. This is always taken at the least likely element. 

When the likelihood order in Corollary \ref{lhoiw} holds an $n$-cycle is the least likely element of the involution walk. The Murnaghan-Nakayama based recursive formula for the eigenvalues of the walk will lead to an explicit formula for $P^{*t}(n)$. Assuming that the $n$-cycle is the least likely element at all times, this gives an explicit computation of the separation distance.

\begin{conjecture}
For $p \geq \frac{1}{2}$ the likelihood order for involution walk is at all times the cycle lexicographic order.
\end{conjecture}

Recall the formula for the eigenvalues of the walk from Proposition \ref{eigformula}:
\[ \psi_\lambda = \sum_{\rho: \lambda/ \rho = [2]} \psi_\rho \frac{d_\rho}{d_\lambda} + (2p-1)\sum_{\rho: \lambda/ \rho = [1,1]} \psi_\rho \frac{d_\rho}{d_\lambda} + 2p\sum_{\rho: \lambda/ \rho = [1]\cup[1]} \psi_\rho \frac{d_\rho}{d_\lambda}\]

When $\lambda = [n-i,1^i]$ there is only one way to remove each of these shapes. In this case, its possible to recurse all the way to the base cases of $\psi_{[2]} = 1$ and $\psi_{[1^2]} = 2p-1$ (as seen in Proposition \ref{neg1}). Note that this formula does not depend on $p \geq \frac{1}{2}$.

\begin{proposition}\label{n-11ieig}
\[\psi_{[n-i,1^i]} = \sum_{j} \frac{ {n/2 -1 \choose j, \frac{i-j}{2}, \frac{n-i-j-2}{2}}}{{n-1 \choose i}}(2p)^j(2p-1)^{(i-j)/2} + \frac{ {n/2 -1 \choose j, \frac{i-j-1}{2}, \frac{n-i-j-1}{2}}}{{n-1 \choose i}}(2p)^j(2p-1)^{(i-j)/2+1} \]
\end{proposition}
\begin{proof}
If the recursion is allowed to continue down to the base cases of $[2]$ and $[1^2]$, then the ratios of dimensions cancel leaving simply $\frac{1}{d_{[n-i,1^i]}} = \frac{1}{{n-1 \choose i}}$. Its left to count how many ways there are to arrive at each base case. To get to $[2]$, a total of $n-i-2$ blocks must be removed from the first row, and $i$ blocks must be removed from the first column. There are $n/2-1$ recursive steps. Let $j$ be the number of removals of the form $[1] \cup [1]$. This forces $\frac{i-j}{2}$ removals of $[1^2]$ from the first column and $\frac{n-i-j-2}{2}$ removals from the first row (excepting the base case). This gives ${n/2-1 \choose j, \frac{i-j}{2}, \frac{n-i-j-2}{2}}$ ways to arrive at the base case. The $j$ $[1] \cup [1]$ removals each come with coefficient $2p$, and the $\frac{i-j}{2}$ $[1^2]$ removals each have coefficient $2p-1$. The base case of $[2]$ gives the term:
 
\[\sum_{j} \frac{ {n/2 -1 \choose j, \frac{i-j}{2}, \frac{n-i-j-2}{2}}}{{n-1 \choose i}}(2p)^j(2p-1)^{(i-j)/2} \]

If instead the base case is $[1^2]$, a total of $n-i-1$ blocks must be removed from the first row, as well as $i-1$ from the first column. By an analagous argument, this gives:

\[\sum_{j} \frac{ {n/2 -1 \choose j, \frac{i-j-1}{2}, \frac{n-i-j-1}{2}}}{{n-1 \choose i}}(2p)^j(2p-1)^{(i-j)/2+1} \]
\end{proof}

This gives a formula very similar to the right hand side of Proposition \ref{SEAWORLD}. Indeed, the character polynomial gives an exact expression for $\psi_{[n-i,1^i]}$ that is very similar to the left hand side of that equality.

\begin{proposition}
\[ \psi_{[n-i,1^i]} = \frac{1}{{n-1 \choose i}}\sum_{k,l} {n/2 \choose k,l}(-1)^l p^{n/2-l-k} (1-p)^{k+l} {n - 2k - 2l -1 \choose i-2l} \]
\end{proposition}
\begin{proof}
From formula (\ref{chp})
\[ \chi_{[n-i,1^i]}(1^{n-2s},2^s) = q_{[1^i]}(n-2s,s,0,...0) = \sum_{j} {{s \choose j}}  \sum_{P=([1^i]=\rho^0,...,\rho^{j})} (-1)^{ht(P)}q_{\rho^j}(n-2s,0,...,0) \]

The only way to remove $2$-cycles from $[1^i]$ is vertically. $q_{[1^{i-2j}]}(n-2s) = d_{[n-i-2(s-j),1^{i-j}} = {n-2s -1 \choose i-2j}$. Therefore,

\[  \chi_{[n-i,1^i]}(1^{n-2s},2^s) = \sum_j {s \choose j} (-1)^j {n-2s -1 \choose i-2j}\]

Using this in the formula for $\psi$ from (\ref{eigf}) then substituting $j=l$, $s-j = k$ gives:
\begin{align} \psi_{[n-i,1^i]} &= \sum_{s=0}^{n/2} {n/2 \choose s}(1-p)^sp^{n/2-s} \frac{\chi_{[n-i,1^i]}(1^{n-2s},2^s)}{d_{[n-i,1^i]}} \\
& = \sum_{s = 0}^{n/2} {n/2 \choose k,l}(1-p)^{k+l}p^{n/2-k-l} {n-2k-2l \choose i-2l} 
\end{align}
\end{proof}

In the event that $p = \frac{1}{2}$, any vertical removal has a $0$ coefficient. When the first column is longer than the first row, its not possible to remove the entire first row without using vertical removals, since $[1]\cup[1]$ removals take equally from the first row and column. Therefore $\psi_{[n-i,i]} = 0$ for $i \geq \frac{n}{2}$. When the first row is longer than the first column, the eigenvalue reduces to a single term. This gives the sum in Proposition \ref{n-11ieig} for $i \leq \frac{n-1}{2}$ as:

\[ \psi_{[n-i,1^i]} = \frac{{n/2-1 \choose i}}{{n -1 \choose i}} = 2^{-i} \frac{(n-2\left\lfloor i/2 \right\rfloor) \cdots (n-2i+4) (n-2i+2)}{(n-1)(n-3) \cdots (n - 2\left\lfloor i/2 \right\rfloor +1)} \]

\begin{conjecture}\label{invsepc}
For $p = \frac{1}{2}$,
\[ \sep(t) = \sum_{i=1}^{(n-1)/2} (-1)^{i+1} {n-i \choose i} \left( \frac{{n/2-1 \choose i}}{{n -1 \choose i}} \right)^t \]
For $t \geq \log_{2}(n-1)$ the terms in this alternating sum are decreasing in magnitude, so
\[ \sep(\log_2(n) + c) \leq 2^{-c}\]
\end{conjecture}

\section*{Acknowledgements}

Thank you to the NSF for its generous support under grant DMS-1344199. Thanks also to my advisor, Persi Diaconis, for his suggestion of this project and many useful conversations during its undertaking.  

\bibliographystyle{plain}
\bibliography{InvolutionWalksJoAP}

 \end{document}